\documentclass[openright, 11pt]{book}
\usepackage{amsmath,amsfonts,amsthm,amssymb,amscd,hyperref,url}
\usepackage[spanish,english]{babel}
\usepackage[T1]{fontenc}
\usepackage[utf8x]{inputenc}  
\def\classification#1{\def\@class{#1}}
\classification{\null}

\DeclareFontFamily{OT1}{rsfs}{}
\DeclareFontShape{OT1}{rsfs}{n}{it}{<-> rsfs10}{}
\DeclareMathAlphabet{\mathscr}{OT1}{rsfs}{n}{it}

\DeclareMathOperator{\mo}{\,mod}

\DeclareMathOperator{\codim}{codim}
\DeclareMathOperator{\GL}{GL}

\DeclareMathOperator{\diam}{diam}

\DeclareMathOperator{\SL}{SL}

\DeclareMathOperator{\PSL}{PSL}

\DeclareMathOperator{\Stab}{Estab}

\DeclareMathOperator{\tr}{tr}

\DeclareMathOperator{\Cl}{Cl}

\newcommand{\x}{\boldsymbol x}

\newtheorem{thm}{Teorema}[section]
\newtheorem{lem}{Lema}[section]
\newtheorem{coro}{Corolario}[section]
\newtheorem{prop}{Proposici\'on}[section]

\newtheorem{prob}{Ejercicio}[section]

\begin{document}
\selectlanguage{spanish}
\thispagestyle{empty} 
\newpage
\thispagestyle{empty}
\begin{center}
AGRA II: Aritmética, grupos y análisis

\vskip .1cm

An ICTP-CIMPA Research School

\vskip .5cm
\hrule
\vskip 3 cm

{\Large\bf\sc Crecimiento y expansi\'on en $\SL_2$} 

\vskip 2cm

{\huge Harald Andrés Helfgott}

\vskip .3cm

{\large Universit\"at G\"ottingen/CNRS/Université de Paris VI/VII
\\[.2cm]
helfgott@math.univ-paris-diderot.fr}

\vfill
\hrule
\vskip .3cm
UNIVERSIDAD S. ANTONIO ABAD, CUSCO, PERÚ, del 8 al 22 de Agosto de 2015 
\end{center}

\newpage
\thispagestyle{empty}
\noindent

 \vfill


\newpage
\pagestyle{myheadings}\markboth{}{}
\thispagestyle{empty}
\frontmatter

\thispagestyle{myheadings}\markboth{}{}
\parskip 0.14cm
\parindent 0.5cm

\chapter{Prefacio}
\'Esta es una breve introducci\'on al estudio del crecimiento en los
grupos finitos, con $\SL_2$ como ejemplo. Su \'enfasis cae sobre los
desarrollos de la \'ultima d\'ecada, provenientes en parte de la combinatoria.

El texto -- basado en parte en \cite{MR3348442} -- consiste, en esencia, en notas de clase para un curso en la escuela de invierno
AGRA II en la Universidad San Antonio Abad, Cusco, Per\'u (10--21 agosto 2015),
incluyendo algunos ejercicios. El curso fue la primera mitad de una unidad;
la segunda mitad, sobre expansores en conexi\'on al espacio hiperb\'olico,
corri\'o a cargo de M. Belolipetsky.

El t\'opico tiene una intersecci\'on apreciable con varios otros textos,
incluyendo el libro \cite{MR3309986} de T. Tao y las notas \cite{MR3144176} de E. Kowalski. El tratamiento en el Cap\'itulo \ref{chap:inter}
difiere un tanto de \cite{MR3348442}, en la medida que sigue un tratamiento
m\'as global (grupos algebraicos) y menos local (\'algebras de Lie); en \'esto
puede detectarse una influencia de \cite{MR3309986} y \cite{MR3144176}
(y, en \'ultima instancia, \cite{LP}). No parecen haber desventajas o ventajas
decisivas en \'esto; simplemente he tomado la oportunidad de explorar un
formalismo distinto.

Sin duda, \cite{MR3348442}
da m\'as detalles que el texto presente, tanto de tipo hist\'orico como
de tipo puramente matem\'atico; su tema tambi\'en
es m\'as amplio. La meta principal aqu\'i es dar
una introducci\'on concisa, accesible
y en cierto sentido participativa al t\'opico.

Las brev\'isimas introducciones al grupo $\SL_2$ y a la
geometr\'ia algebraica en general tienen como intenci\'on hacer que el texto
comprensible sea comprensible para estudiantes de distintas \'areas,
aparte de ser partes necesarias de la cultura general. Se pide la paciencia
del los lectores para los cuales tales introducciones son innecesarias.


\mainmatter
\pagestyle{myheadings}
\markboth{\sf{Harald Andr\'es Helfgott}}{\sf{Crecimiento y expansi\'on en $\SL_2$}}

\chapter{Introducci\'on}

Nuestra tema es el crecimiento en los grupos; nuestro ejemplo
ser\'an los grupos $\SL_2(K)$, $K$ un cuerpo finito.

Qu\'e se quiere decir aqu\'i por {\em crecimiento}? Hay diferentes puntos
de vista, dependiendo del \'area.
La manera m\'as concreta de
expresar la cuesti\'on es quiz\'as la siguiente: tenemos un subconjunto finito
$A$ de un grupo $G$. Consideremos los conjuntos
\[\begin{aligned}
A&,\\
A\cdot A &= \{x\cdot y : x,y\in A\},\\
A\cdot A\cdot A &= \{x\cdot y\cdot z: x,y,z\in A\},\\
&\dotsc\\
A^k &= \{x_1 x_2 \dotsc x_k : x_i\in A\}.
\end{aligned}\]
Escribamos $|S|$ por el n\'umero de elementos de un conjunto finito $S$. La
pregunta es: qu\'e tan r\'apido crece $|A^k|$ a medida que $k$ se
incrementa?

Tal cuesti\'on ha sido estudiada desde la perspectiva de la combinatoria
aditiva (caso de $G$ abeliano) y de la teor\'ia de grupos geom\'etrica 
($G$ infinito, $k\to \infty$). También hay varios conceptos relacionados,
de suma importancia, provenientes de la teor\'ia de grafos y de analog\'ias
con la geometr\'ia: {\em di\'ametros}, {\em expansores}, etc.

Ahora bien, porqué elegir a los grupos $\SL_2(K)$ como primer caso a 
estudiar, m\'as all\'a de la necesidad, en una exposici\'on, de comenzar por
un caso concreto?

\section{Los grupos $\SL_2(R)$}

Sea $R$ un anillo; por ejemplo, podemos tomar $R=\mathbb{Z}$, o
$R=\mathbb{Z}/p\mathbb{Z}$. Definimos
\[\SL_2(R) = \left\{\left(\begin{matrix}a & b\\c & d\end{matrix}\right):
a,b,c,d\in R,\;\; a d - b c = 1.
\right\}.\]
La letra ``S'' en $\SL$ viene de e{\bf s}pecial (lo que aqu\'i quiere decir:
 de determinante igual a $1$), mientras que ``L'' viene de {\bf l}ineal
(por tratarse de un grupo de matrices). El n\'umero $2$ viene del hecho 
 que \'estas son matrices $2$x$2$.

El grupo $\SL_2(R)$ puede verse por lo menos de dos formas: como un grupo 
abstracto, y
como un grupo de transformaciones geom\'etricas. Visto de una manera o
la otra, se trata de un buen caso a estudiar, pues es, por as\'i decirlo,
el objeto m\'as sencillo en demostrar toda una gama de comportamientos 
complejos. Veamos c\'omo.

\subsection{La estructura del grupo $\SL_2(R)$}

Dado un grupo $G$, nos interesamos en sus subgrupos $H<G$, y, en particular,
en sus subgrupos {\em normales} $H\triangleleft G$.
Dado $H\triangleleft G$, podemos decir que $G$ se descompone en $H$
y $G/H$. Un grupo $G$ sin subgrupos normales (aparte de $\{e\}$ y $G$)
se llama {\em simple}. 

Los grupos simples juegan un rol similar al de los primos en los enteros.
Es f\'acil ver que, para todo grupo finito $G$, existen
\begin{equation}\label{eq:turu}\{e\} = H_0 \triangleleft H_1 \triangleleft H_2 \triangleleft \dotsb 
\triangleleft H_k = G\end{equation}
tales que $H_i/H_{i-1}$ es simple y no-trivial para $1\leq i\leq k$.
El teorema de Jordan-H\"older nos dice que tal decomposici\'on es en
esencia \'unica: los factores $H_{i+1}/H_i$ en (\ref{eq:turu}) est\'an
determinados por $G$, y a lo m\'as su orden puede cambiar.

Un grupo resoluble es un grupo que tiene una decomposici\'on tal que
$H_{i+1}/H_i$ es abeliano para todo $i$. 
Como hemos dicho, la combinatoria aditiva ha estudiado tradicionalmente el
crecimiento en los grupos abelianos.
El estudio del crecimiento en los grupos resolubles esta lejos de ser
trivial, o de reducirse por completo al crecimiento en los grupos
abelianos. Empero, una decomposici\'on $H\triangleleft G$ reduce los
problemas de crecimiento (como muchos otros) al estudio de (a) los grupos
$H$ y $G/H$, (b) la {\em acci\'on} de $G/H$ sobre $H$. Por ello, en \'ultimas cuentas, tiene sentido concentrarse en el estudio de los grupos simples, y,
en particular, en el estudio del crecimiento en los grupos simples no abelianos.

Sea $K$ un cuerpo finito. El grupo $\SL_2(K)$ no es ni abeliano ni resoluble.
El {\em centro}
\[Z(G) = \{g\in G: \forall h\in G\;\;\;\;h g = g h\}\]
de un grupo $G$ es siempre un subgrupo normal de $G$. Ahora bien,
para $G = \SL_2(K)$, $Z(G)$ es igual a $\{I,-I\}$; as\'i, a menos que $K$ sea el
cuerpo $\mathbb{F}_2$ con dos elementos, $Z(G)$ no es el grupo trivial
$\{e\}\ne \{I\}$, y por lo tanto
$G=\SL_2(K)$ no es simple. Empero, el cociente
\[\PSL_2(K) := \SL_2(K)/Z(\SL_2(K)) = \SL_2(K)/\{I,-I\}\]
s\'i es simple, para $|K|$ finito y mayor que $3$.\footnote{Este es un
hecho no trivial. Para $|K|$ un n\'umero primo $>3$, fue mencionado, pero no
probado, por Galois (1831). Fue probado para tales $K$ por Jordan
\cite{zbMATH02675601}, y para $K$ finito general por Moore
\cite{zbMATH02721437}. La nota \cite{Conrad} contiene tanto estos apuntes
hist\'oricos como una prueba completa para $\SL_n(K)$, donde $n>2$ o
$n=2$ y $|K|>3$.}

{\small
{\em Comentario de \'indole cultural.}
As\'i, al considerar $\SL_2(K)$ para $K$ variando sobre todos los cuerpos
finitos, obtenemos toda un conjunto infinito de grupos finitos simples
$\PSL_2(K)$. Se trata, por as\'i decirlo, de la familia m\'as sencilla
de grupos finitos simples, junto con aquella dada por los grupos
{\em alternantes} $A_n$. (El grupo $A_n$ es el \'unico subgrupo de \'indice $2$ 
del {\em grupo sim\'etrico} $S_n$, el cual, a su vez,  consiste en las $n!$ permutaciones de $n$ elementos,
con la composici\'on como operaci\'on del grupo.) En verdad, el famoso Teorema
de la Clasificaci\'on de grupos simples nos dice que hay dos tipos
de familias infinitas de grupos simples: las familias de grupos de matrices,
como $\PSL_2(K)$, y la familia $A_n$, aparte de un n\'umero finito de
grupos especiales (como el as\'i llamado ``monstruo'').}

Veamos la estructura de $G=\SL_2(K)$ en m\'as detalle. Si bien $\SL_2(K)$
no tiene subgrupos normales m\'as all\'a de $\{I\}$, $\{I,-I\}$ y
$\SL_2(K)$, tiene subgrupos de varios otros tipos. Los m\'as interesantes para
nosotros son los {\em toros}; en $\SL_2(K)$, aparte del grupo trivial, todos
son {\em toros m\'aximos}. Recordamos que el centralizador de un
elemento $g\in G$
es el grupo
\begin{equation}\label{eq:pastrul}
  C(g) = \{h\in G: hg = g h\}.\end{equation}
Un {\em toro m\'aximo} (denotado
por $T(K)$) es un grupo $C(g)$ donde $g$ es
{\em regular semisimple}; un elemento $g\in G$ es
{\em regular semisimple} si tiene dos valores propios distintos.
Esto es lo mismo que decir que $T(K) = \sigma D \sigma^{-1} \cap
\SL_2(K)$, donde $D$ es el grupo de matrices diagonales
\[D =  \left\{\left(\begin{matrix}r & 0\\0 & r^{-1}\end{matrix}\right):
r\in \overline{K} \right\}
\]
y $\sigma \in \SL_2(\overline{K})$. Aqu\'i $\overline{K}$ es la compleci\'on
(clausura) algebraica de $K$. N\'otese que $\sigma$ puede o puede no
estar en $\SL_2(K)$.

{\em Ejemplo.} Sea $K = \mathbb{R}$. Si $\sigma\in \SL_2(\mathbb{R})$,
entonces $T(K) = \sigma D \sigma^{-1} \cap
\SL_2(K)$ es de la forma
\[
\sigma \left\{\left(\begin{matrix}r & 0\\0 & r^{-1}\end{matrix}\right): 
r\in \mathbb{R}^*\right\}  \sigma^{-1},\]  
y es m\'as bien de la forma
\begin{equation}\label{eq:toro}
  \tau 
\left\{\left(\begin{matrix}\cos \theta & -\sin \theta\\ \sin \theta &
  \cos \theta\end{matrix}\right) : \theta \in \mathbb{R}/\mathbb{Z}\right\}\tau^{-1}\end{equation} (para alg\'un $\tau
  \in \SL_2(\mathbb{R})$) si $\sigma\not\in \SL_2(\mathbb{R})$.
  Est\'a claro que, en el segundo caso, $T(K)$ es un c\'irculo, es decir,
  un toro $1$-dimensional (en el sentido tradicional de ``toro'').
  

\subsection{El grupo de transformaciones $\SL_2(\mathbb{R})$} 

Uno de varios modelos equivalentes para la geometr\'ia hiperb\'olica en dos
dimensiones es el semi-plano de Poincar\'e, tambi\'en llamado simplemente
 semi-plano superior:
\[\mathbb{H} = \{(x,y) \in \mathbb{R}^2: y>0\}\]
con la m\'etrica dada por 
\[ds = \frac{\sqrt{dx^2 + dy^2}}{y}.\]
Las isometr\'ias de $\mathbb{H}$ que preservan la orientaci\'on
son las transformaciones lineares fraccionales
\begin{equation}\label{eq:housecodd}
z\mapsto \frac{az+b}{cz+d},\end{equation}
donde $a,b,c,d\in \mathbb{R}$ y $ad-bc\ne 0$. Es f\'acil ver que
esto induce una biyecci\'on $\PSL_2(\mathbb{R})$ al conjunto de
transformaciones lineares fraccionales. Escribamos $g z$ para la imagen
de $z$ bajo una transformaci\'on (\ref{eq:housecodd}) inducida por una
matriz correspondiente a un elemento $g\in \PSL_2(\mathbb{R})$. No es nada
dif\'icil verificar que, para $g_1, g_2  \in \PSL_2(\mathbb{R})$,
\[g_1(g_2 z) = (g_1 g_2) z.\]
En otras palabras, tenemos un isomorfismo de $\PSL_2(\mathbb{R})$ al grupo
(con la composici\'on como operaci\'on)
de las transformaciones lineares fraccionales.

Podemos considerar subacciones; por ejemplo, el {\em grupo modular (completo)}
$\SL_2(\mathbb{Z})$ act\'ua sobre $\mathbb{H}$. El cociente
$\SL_2(\mathbb{Z})\backslash \mathbb{H}$ 
es de vol\'umen finito sin ser compacto.
Tambien podemos considerar los {\em grupos modulares de congruencia}
\[\Gamma(N) = 
\left\{\left(\begin{matrix}a & b\\c & d\end{matrix}\right)\in \SL_2(\mathbb{Z}):
a\equiv d\equiv 1 \mo N,\;\;\; b\equiv c\equiv 0 \mo N
\right\}\]
para $N\geq 1$. Claro est\'a, $\Gamma(N)$ es el n\'ucleo (``kernel'') de la
reducci\'on $\SL_2(\mathbb{Z})\to \SL_2(\mathbb{Z}/N\mathbb{Z})$,
y, por lo tanto,
 $\Gamma(N)\backslash \mathbb{H}$ consiste de
$|\SL_2(\mathbb{Z}/N\mathbb{Z})|$ copias de
 $\SL_2(\mathbb{Z})\backslash \mathbb{H}$.

\section{Perspectivas sobre el crecimiento en los grupos}\label{sec:dfsq}

El crecimiento en los grupos ha sido estudiado de varias perspectivas 
distintas. Nos concentraremos despu\'es en desarrollos relativamente 
recientes que toman sus herramientas en parte de algunas de estas \'areas
(clasificaci\'on de subgrupos, combinatoria aditiva) y su relevancia de otras
(el estudio de los di\'ametros y la expansi\'on). Hay a\'un otras areas
de suma importancia cuya relaci\'on con nuestro tema reci\'en comienza a 
elucidarse (teor\'ia de modelos, teor\'ia de grupos geom\'etrica). Demos
una mirada a vuelo de p\'ajaro.

{\em Combinatoria aditiva.} \'Este es en verdad un nombre reciente para
un campo de estudios m\'as antiguo, con una cierta intersecci\'on
con la {\em teor\'ia aditiva de los n\'umeros}. Se puede decir que la combinatoria aditiva se diferencia de esta \'ultima cuando considera el crecimiento
de conjuntos bastante arbitrarios, y no s\'olamente el de conjuntos como
los primos o los cuadrados. Uno de los resultados claves es el teorema
de Freiman \cite{MR0360496}, el cual clasifica los subconjuntos finitos 
$A\subset \mathbb{Z}$ tales que $A+A$ no es mucho m\'as grande que $A$.
Ruzsa dio una segunda prueba \cite{MR1139055}, m\'as general y m\'as simple,
e introdujo muchos conceptos ahora claves en el \'area.

El uso del signo $+$ y de la palabra {\em aditiva} muestran que,
hasta recientemente, la combinatoria aditiva estudiaba grupos abelianos,
si bien algunas de sus t\'ecnicas se generalizan a los grupos no abelianos
de manera natural. 

{\em Clasificaci\'on de subgrupos.}
Sea $A$ un subconjunto de un grupo $G$. Asumamos que $A$ contiene la identidad
$e\in G$. Entonces $|A\cdot A| = |A|$ s\'i y s\'olo s\'i $A$ es un subgrupo
de $G$; en otras palabras, clasificar los subgrupos de $G$ equivale a
clasificar los subconjuntos de $G$ que no crecen.

Clasificar los subgrupos de un grupo es a menudo algo lejos de trivial --
m\'as a\'un si se desea emprender tal tarea sin utilizar la Clasificac\'ion
de los grupos simples (una herramienta muy fuerte, cuya prueba
fue inicialmente juzgada incompleta o poco satisfactoria por muchos).
Resulta ser que los trabajos en este \'area basados sobre argumentos
elementales, antes que sobre la Clasificaci\'on, son a veces robustos: pueden
ser adaptados para darnos informaci\'on, no s\'olo sobre los subgrupos
de $G$, sino sobre los subconjuntos $A$ de $G$ que crecen poco.

{\em Di\'ametros y tiempos de mezcla.} 
Sea $A$ un conjunto de generadores de un grupo $G$; en otras palabras,
$A\subset G$ es tal que todo elemento $g$ de $G$ puede escribirse como un
producto $g = x_1 x_2 \dotsc x_r$ para alguna elecci\'on de $x_i\in A\cup A^{-1}$.
El {\em di\'ametro} de $G$ con respecto a $A$
es el $k$ m\'inimo tal que todo elemento
$g$ de $G$ puede escribirse como $g = x_1 x_2 \dotsc x_r$ con $x_i\in A$
y $r\leq k$. Si $G$ es finito, el di\'ametro es necesariamente finito.

Por qu\'e hablamos de ``di\'ametro''? El {\em gr\'afo de Cayley}
$\Gamma(G,A)$ es el grafo que tiene $G$ como su conjunto de v\'ertices
y $\{(g, a g): g\in G, a\in A\}$ como su conjunto de aristas. Podemos
definir la distancia $d(g_1,g_2)$
entre $g_1, g_2\in G$ como la longitud del camino
m\'as corto de $g_1$ a $g_2$, donde se define que la longitud de cada
arista es $1$. Definimos el di\'ametro de un grafo como definimos el
de cualquier figura: ser\'a el m\'aximo de la distancia $d(g_1,g_2)$
para toda elecci\'on posible de vertices $g_1$, $g_2$. Es f\'acil verificar
que $\diam(\Gamma(G,A))$ es igual al di\'ametro de $G$ con respecto a $A$
que acababamos de definir.

La conjetura de Babai \cite[p.~176]{BS88} postula que, si $G$ es simple
y no abeliano, entonces, para cualquier conjunto de generadores $A$ de $G$,
\[\diam(\Gamma(G,A))\ll (\log |G|)^{O(1)},\]
donde las constantes impl\'icitas son absolutas.

(Un poco de notaci\'on. Sean $f, g$ funciones de un conjunto $X$ a $\mathbb{C}$.
Como es habitual en la teor\'ia anal\'itica de n\'umeros, para nosotros,
  $f(x)\ll g(x)$, $g(x)\gg f(x)$ y $f(x) = O(g(x))$ quieren decir la misma
cosa: hay $C>0$ y $X_0\subset X$ finito (``constantes impl\'icitas'')
  tales que $|f(x)|\leq C\cdot g(x)$ para todo
  $x\in X$ fuera de $X_0$. (En verdad, necesitamos que $f(x)$ y $g(x)$ estén bien
  definidas s\'olo para $x$ fuera de $X_0$.)
    Escribimos $\ll_a$, $\gg_a$, $O_a$ si $X_0$ y $C$ dependen de
  una cantidad $a$ (digamos). Si $X_0$ y $C$ no dependen de nada, las llamamos constantes
  {\em absolutas}.)

  El {\em tiempo de mezcla} es el $k$ m\'inimo tal que, si
  $x_1,x_2,\dotsc,x_k$ son tomados al azar en $A$ con la distribuci\'on
  uniforme en $A$, la distribuci\'on del producto $x_1 \dotsb x_k$
  (o, lo que es lo mismo, la distribuci\'on
  del resultado de una caminata aleatoria de longitud $k$ en
  $\Gamma(G,A)$) esta cerca de la distribuci\'on uniforme en $G$. 
  Hablamos de distintos tiempos de mezcla dependiendo de lo que se quiera decir
  por ``cerca''. El estudio de los tiempos de mezcla ha tenido no solo
  un fuerte color probabil\'istico (v\'eanse las referencias \cite{MR1245303},
  \cite{MR2466937})
  sino a menudo tambi\'en algor\'itmico (e.g. en \cite{BBS04}).
  
{\em Expansores y huecos espectrales.}
Comenzemos dando una definici\'on elemental de lo que es un expansor.
Sea $A$ un conjunto de generadores de un grupo finito $G$. Decimos que
el grafo $\Gamma(G,A)$ es un $\epsilon$-expansor (para $\epsilon>0$ dado)
si todo subconjunto $S\subset G$ con $|S|\leq |G|/2$ satisface
$|S\cup AS|\geq (1+\epsilon) |S|$. Es muy simple de ver
que todo $\epsilon$-expansor tiene di\'ametro 
$O((\log |G|)/\epsilon)$, es decir, muy peque\~no; est\'a 
claro que el di\'ametro de
$\Gamma(G,A)$ es siempre por lo menos $O((\log |G|)/(\log |A|))$.

La alternativa (al final equivalente) es definir los grafos expansores en
t\'erminos del primer valor propio no trivial $\lambda_1$ del Laplaciano
discreto de un grafo de Cayley. La {\em matriz de adyacencia} (normalizada)
$\mathcal{A}$ de un grafo es un operador linear en el espacio de funciones
$f:G\to \mathbb{C}$; env\'ia tal funci\'on a la func\'ion cuyo valor
en $v$ es el promedio de $f(w)$ en los vecinos $w$ de $v$.
Para ser expl\'icitos, en el caso del grafo de Cayley $\Gamma(G,A)$,
\begin{equation}\label{eq:dudur}
  (\mathscr{A} f)(g) = \frac{1}{|A|} \sum_{a\in A} f(a g).\end{equation}
El {\em Laplaciano discreto} es simplemente $\triangle = I-\mathscr{A}$.
(Muchos lectores lo reconocer\'an como el
an\'alogo de un Laplaciano sobre una superficie.)

Asumamos $A = A^{-1}$. Entonces $\triangle$ es un operador sim\'etrico,
as\'i que todos sus valores propios son reales. Est\'a claro que
el valor propio m\'as peque\~no es $\lambda_0 = 0$, correspondiente
a las funciones propias constantes. Podemos ordenar los valores propios:
\[0 = \lambda_0 \leq \lambda_1 \leq \lambda_2 \leq \dotsc.\]
A la cantidad $|\lambda_1 - \lambda_0| = \lambda_1$ se le da el nombre
  de {\em hueco espectral}.
  
Decimos que $\Gamma(G,A)$ es un $\epsilon$-expansor si
$\lambda_1\geq \epsilon$. Para $|A|$ acotado, esta definici\'on es
equivalente a la primera que dimos (si bien la constante $\epsilon$
difiere en las dos definiciones). 
Decimos que una familia (conjunto infinito) de grafos $\Gamma(G,A)$
es una {\em familia de expansores}, o que es una familia con un 
hueco espectral, si todo grafo en la familia es un $\epsilon$-expansor
para alg\'un $\epsilon>0$ fijo.

Uno de los problemas centrales del \'area es probar que ciertas
familias (si no todas las familias) del tipo
\[\{\Gamma(\SL_2(\mathbb{Z}/p\mathbb{Z}),A_p)\}_{\text{$p$ primo}},
\;\;\;\;\;\;\;\;\;\text{$A_p$ genera
  $\SL_2(\mathbb{Z}/p\mathbb{Z})$}
\]
son familias de expansores. Los primeros resultados atacaban el
problema mediante el estudio del Laplaciano sobre las superficies
$\Gamma(p)\backslash \mathbb{H}$. Un resultado cl\'asico
de Selberg \cite{MR0182610} nos dice
que el Laplaciano sobre
$\Gamma(p)\backslash \mathbb{H}$ tiene un hueco espectral independiente
de $\epsilon$.

{\em Teor\'ia de grupos geom\'etrica. Teor\'ia de modelos.}
La teor\'ia de grupos geom\'etrica se centra en el estudio del crecimiento
de $|A^k|$ para $k\to \infty$, donde $A$ es un subconjunto de un grupo
infinito $G$. Por ejemplo, un teorema de Gromov \cite{MR623534}
muestra que, si $A$ genera a $G$
y $|A^k| \ll k^{O(1)}$, entonces $G$ tiene que ser
un grupo de un tipo muy particular (virtualmente nilpotente, para ser precisos).
Los argumentos de la teor\'ia de grupos geom\'etrica a menudo muestran
que, a\'un si un grupo no est\'a dado a priori de una manera geom\'etrica,
el crecimiento de un subconjunto puede darle de manera natural una geometr\'ia
que puede ser utilizada.

Si bien los problemas tratados por la teor\'ia de grupos geom\'etrica
son muy cercanos a los nuestros,
tales argumentos a\'un no son moneda corriente en el sub\'area que discutiremos
en estas notas, lo cual puede decir simplemente que la manera de aplicarlos
a\'un est\'a
por descubrirse. Por otra parte, la {\em teor\'ia de modelos}
-- en esencia, una rama de la l\'ogica con aplicaciones a las estructuras
algebraicas -- ha jugado un rol directo en el sub\'area. Por ejemplo,
Hrushovski \cite{MR2833482} dio una nueva prueba del teorema de Gromov,
expresando cuestiones del crecimiento en grupos en un lenguaje proveniente
de la teor\'ia de modelos; m\'as all\'a en esta direcci\'on, se debe
mencionar a \cite{BGTstru}. Se trata de temas que parecen estar lejos de
estar agotados.

\section{Resultados}\label{sec:res}

Uno de los prop\'ositos es dar una prueba del siguiente
resultado, debido al autor \cite{Hel08} para $K = \mathbb{Z}/p\mathbb{Z}$.
No ser\'a id\'entica a la primera prueba que diera, sino
que incluir\'a las ideas de varios autores posteriores, inclu\'idas
algunas que han hecho que el enunciado sea m\'as general que
el original, y otras que han hecho
que la prueba sea m\'as clara y
f\'acil de generalizar. En cualquier forma, el enunciado deriva claramente su
inspiraci\'on de la combinatoria aditiva.

\begin{thm}\label{thm:main08}
  Sea $K$ un cuerpo. Sea $A\subset \SL_2(K)$ un conjunto que genera
  $\SL_2(K)$.
  Entonces, ya sea
  \[|A^3| \geq |A|^{1+\delta}\]
  o $(A\cup A^{-1}\cup \{e\})^k  = \SL_2(K)$, donde $\delta>0$ y $k>0$ son
  constantes absolutas. 
\end{thm}
Por cierto, gracias a \cite{MR2410393} y \cite{MR2800484},
$(A \cup A^{-1} \cup \{e\})^k = \SL_2(K)$ puede reemplazarse por $A^3$.
Mostraremos por lo menos que podemos tomar
$k=3$. Las primeras generalizaciones
a $K$ de orden finito no primo 
se deben a \cite{MR2788087} y \cite{MR2862040}; hoy en d\'ia, se obtiene la
forma general sin mayores complicaciones.

Veamos una consecuencia sencilla.
\begin{prob}
Sea $G = \SL_2(K)$, $K$ un cuerpo finito. Sea $A\subset G$ un conjunto que genera $G$.
El di\'ametro de $\Gamma(G,A)$ es $\ll (\log |G|)^{O(1)}$, donde
las constantes impl\'icitas son absolutas.
\end{prob}
Este enunciado es exactamente la conjetura de Babai para $G = \SL_2(K)$.

Cu\'ando es que $\Gamma(G,A)$ tiene di\'ametro $\ll \log |G|$?
Yendo m\'as lejos -- cu\'ando es un $\epsilon$-expansor?

Se sab\'ia desde los a\~{n}os 80 (ver las referencias en \cite{MR3348442})
que la existencia de un agujero
espectral para $\Gamma(p)\backslash \mathbb{H}$ (probada en
\cite{MR0182610}) implica que, para
\begin{equation}
A_0 = \left\{ \left(\begin{array}{cc} 1 &1\\0 &1\end{array}\right),
\left(\begin{array}{cc} 1 &0\\1 &1\end{array}\right)
\right\},\end{equation}
los gr\'afos $\Gamma(\SL_2(\mathbb{Z}/p\mathbb{Z}),A_0 \mo p)$ forman
una familia de expansores (i.e., son todos $\epsilon$-expansores
para alg\'un $\epsilon$ fijo).
Empero, para, digamos,
\begin{equation}\label{eq:relel}
A_0 = \left\{ \left(\begin{array}{cc} 1 &3\\0 &1\end{array}\right),
\left(\begin{array}{cc} 1 &0\\3 &1\end{array}\right)
\right\},\end{equation}
no se ten\'ia tal resultado, ni a\'un una cota razonable para el
di\'ametro. (\'Esta diferencia fue resaltada varias veces por Lubotzky.)

\begin{prob}\label{ej:luk}
  Sea $G = \SL_2(\mathbb{Z}/p\mathbb{Z})$; sea $A_0$ como en
  (\ref{eq:relel}). Pruebe que el di\'ametro de $\Gamma(G,A_0 \mo p)$
  es $\ll \log |G|$. 
\end{prob}
Para resolver este ejercicio, es \'util saber que $A_0$ genera
un subgrupo libre de $\SL_2(\mathbb{Z})$. Decimos que un conjunto $A_0$
en un grupo genera un {\em subgrupo libre}
si no existen $x_i\in A_0$, $1\leq i\leq k$, $x_{i+1}\notin \{x_i,x_i^{-1}\}$
para
$1\leq i\leq k-1$, $x_i\ne e$ para $1\leq i\leq k$,
y $r_i\in \mathbb{Z}$, $r_i\ne 0$, tales que
\[x_1^{r_1} \dotsb x_k^{r_k} = e.\]
Saber que $A_0$ genera un subgrupo libre
es particularmente \'util en los primeros pasos de la iteraci\'on; en los ultimos pasos, podemos utilisar el Teorema~\ref{thm:main08}. 

En verdad, la aseveraci\'on del ejercicio \ref{ej:luk} sigue siendo 
v\'alida sin la suposici\'on
que el grupo $\langle A_0\rangle$ generado por $A_0$ sea libre;
es suficiente (y f\'acil) mostrar que $\langle A_0\rangle$ siempre tiene un
subgrupo libre grande (ver el ap\'endice \ref{chap:appa}).
S\'i se debe asumir que
$\langle A_0\rangle$ genera un subgrupo {\em Zariski-denso} de $\SL_2$, para
as\'i asegurar que
$A_0 \mo p$ en verdad genere $\SL_2(\mathbb{Z}/
p\mathbb{Z})$, para $p$ mayor que una constante $C$.

Bourgain y Gamburd
\cite{MR2415383} fueron netamente m\'as lejos: probaron que, si
$A_0$ genera un grupo Zariski-denso de $\SL_2$, entonces
  \[\{\Gamma(\SL_2(\mathbb{Z}/p\mathbb{Z}),A_0 \mo p)\}_{\text{$p>C$, $p$ primo}}\]
es una familia de expansores.
 Nos concentraremos en dar una prueba del Teorema \ref{thm:main08}. Al final,
 esbozaremos el procedimiento de Bourgain y Gamburd, basado en parte sobre
 dicho teorema.

Queda a\'un mucho por hacer; por ejemplo, no sabemos si la familia
de todos los grafos
\[\{\Gamma(\SL_2(\mathbb{Z}/p\mathbb{Z}),A)\}_{\text{$p$ primo,
    $A$ genera $\SL_2(\mathbb{Z}/p\mathbb{Z})$}}\]
es una familia de expansores.
Por otra parte, si bien hay generalizaciones del teorema
\ref{thm:main08} a otros grupos lineares
(\cite{HeSL3}, \cite{GH1}, y, de manera m\'as general, \cite{BGT} y \cite{PS}), a\'un no tenemos una prueba
de la conjetura de Babai para los grupos alternantes $A_n$; la mejor
cota conocida para el di\'ametro de $A_n$ con respecto a un conjunto
arbitrario de generadores es la cota dada en \cite{MR3152942}, la
cual no es tan buena como $\ll (\log |G|)^{O(1)}$.


\chapter{Herramientas elementales}
\section{Productos triples}
La combinatoria aditiva, al estudiar el crecimiento, estudia los
conjuntos que crecen lentamente. En los grupos abelianos, sus resultados
son a menudo enunciados de tal manera que clasifican los conjuntos $A$
tales que $|A^2|$ no es mucho m\'as grande que $|A|$; en los grupos no-abelianos,
generalmente se clasifica los conjuntos $A$ tales que $|A^3|$ no es mucho
m\'as grande que $|A|$. Por qué?

En un grupo abeliano, si
$|A^2| < K |A|$, entonces $|A^k| < K^{O(k)} |A|$ -- i.e., si un conjunto
no crece después de ser multiplicado por si mismo una vez, no crecera
después de ser multiplicado por s\'i mismo muchas veces. \'Este es un
resultado de Pl\"unnecke \cite{MR0266892} y Ruzsa \cite{MR2314377};
Petridis \cite{MR3063158} dio recientemente una prueba particularmente elegante.

En un grupo no abeliano, puede haber conjuntos $A$ que rompen esta regla.
\begin{prob}
  Sea $G$ un grupo. Sean $H<G$, $g\in G\setminus H$ y
  $A = H \cup \{g\}$.
Entonces $|A^2| < 3 |A|$, pero $A^3\supset H g H$, y
$H g H$ puede ser mucho m\'as grande que $A$. Dé un ejemplo con
$G=\SL_2(\mathbb{Z}/p\mathbb{Z})$.
  \end{prob}

Empero, las ideas de Ruzsa s\'i se aplican al caso no abeliano, como fue
indicado en  \cite{Hel08} y \cite{MR2501249}; en verdad,
no hay que cambiar nada en \cite{MR810596}, pues nunca utiliza la
condici\'on que $G$ sea abeliano. Lo que obtendremos es que
basta con que $|A^3|$ (en vez de $|A^2|$) no sea mucho m\'as grande que
$|A|$ para que $|A^k|$ crezca lentamente.
Veamos como se hacen las cosas.

\begin{lem}[Desigualdad triangular de Ruzsa]\label{lem:schatte}
  Sean $A$, $B$ y $C$ subconjuntos finitos de un grupo $G$. Entonces
  \begin{equation}\label{eq:eolt}
|A C^{-1}| |B| \leq |A B^{-1}| |B C^{-1}| .\end{equation}
\end{lem}
\begin{proof}
Contruiremos una inyecci\'on
$\iota:A C^{-1} \times B \hookrightarrow A B^{-1} \times
B C^{-1}$.
Para cada
$d\in A C^{-1}$, escojamos $(f_1(d),f_2(d)) = (a,c)\in A\times C$ tal que
$d = a c^{-1}$. Definamos $\iota(d,b) = (f_1(d) b^{-1}, b (f_2(d))^{-1})$.
Podemos recuperar $d = f_1(d) (f_2(d))^{-1}$ de $\iota(d,b)$; por lo tanto,
podemos recuperar $(f_1,f_2)(d)=(a,c)$, y as\'i tambi\'en $b$. Por lo tanto,
 $\iota$ es una inyecci\'on.
\end{proof}

\begin{prob}
  Sea $G$ un grupo.
Pruebe que
\begin{equation}\label{eq:mony}
  \frac {|(A \cup A^{-1} \cup \{e\})^3|}{|A|} \leq \left(3\frac{|A^3|}{|A|}\right)^3\end{equation}
  para todo conjunto finito $A$ de $G$. Muestre también que, si $A = A^{-1}$
  (i.e., si $g^{-1}\in A$ para todo $g\in A$), entonces
  \begin{equation}\label{eq:jotor}
    \frac{|A^k|}{|A|} \leq \left(\frac{|A^3|}{|A|}\right)^{k-2}.\end{equation}
  para todo $k\geq 3$. Concluya que
  \begin{equation}\label{eq:marmundo}\frac{|A^k|}{|A|} \leq
\frac{|(A\cup A^{-1} \cup \{e\})^k|}{|A|} \leq
3^{k-2} \left(\frac{|A^3|}{|A|}\right)^{3(k-2)}\end{equation}
  para todo $A\subset G$ y todo $k\geq 3$.
\end{prob}

Esto quiere decir que, de ahora en adelante, si obtenemos que $|A^k|$
no es mucho m\'as grande que $|A|$, podemos concluir que $|A^3|$ no
es mucho m\'as grande que $|A|$. Por cierto, gracias a (\ref{eq:mony}),
podremos suponer en
varios contextos que $e\in A$ y $A = A^{-1}$ sin pérdida de generalidad.

\section{El teorema de \'orbita-estabilizador para los conjuntos}
Una de las ideas recurrentes en la investigaci\'on del crecimiento en los grupos
es la siguiente: muchos enunciados acerca de los subgrupos -- as\'i como
sus métodos de prueba -- pueden generalizarse a los subconjuntos. Si el
método de prueba es constructivo, cuantitativo o probabil\'istico, esto es
un indicio que la prueba podr\'ia generalizarse de tal manera.

El {\em teorema de \'orbita-estabilizador} es un buen ejemplo, tanto
por su simplicidad (realmente deber\'ia llamarse ``lema'')
como por subyacer a un n\'umero sorpredente de resultados sobre el crecimiento.

Primero, un poco de lenguaje. Una {\em acci\'on} $G\curvearrowright X$
es un homomorfismo de un grupo $G$ al grupo de automorfismos de un
objeto $X$. Estudiaremos el caso en el que $X$ es simplemente un conjunto;
su ``grupo de automorfismos'' es simplemente el grupo de biyecciones de $X$
a $X$ (con la composici\'on como operaci\'on de grupo.) Para 
$A\subset G$ y $x\in X$, la {\em \'orbita} $A x$ (``\'orbita de $x$ bajo la
acci\'on de $A$'')
es el conjunto
$A x = \{g\cdot x : g\in A\}$. El {\em estabilizador} $\Stab(x)\subset G$ 
est\'a dado por $\Stab(x) = \{g\in G: g\cdot x = x\}$.

El enunciado que daremos es como en \cite[\S 3.1]{MR3152942}.
\begin{lem}[Teorema de \'orbita-estabilizador para conjuntos]\label{lem:orbsta}
  Sea $G$ un grupo actuando sobre un conjunto $X$.
  Sea $x\in X$, y sea $A\subseteq G$ no vac\'io.
Entonces
\begin{equation}\label{eq:applepie}
|(A^{-1} A) \cap \Stab(x)|\geq \frac{|A|}{|A x|}\end{equation}
y, para $B\subseteq G$,
\begin{equation}\label{eq:easypie}
  |B A| \geq |A \cap \Stab(x)| |B x| .
  \end{equation}
\end{lem}
El teorema de \'orbita-estabilizador usual, que se ense\~na usualmente en un primer curso de teor\'ia de grupos dice que, para $H$ un subgrupo de $G$,
\[|H\cap \Stab(x)| = \frac{|H|}{|H x|}.\]
\'Este es un caso especial del lema que estamos por probar --
el caso $A = B = H$.
\begin{prob}
  Pruebe el Lema \ref{lem:orbsta}. Sugerencia: 
  para (\ref{eq:applepie}), use el principio de los palomares.
\end{prob}

El grupo $G$ tiene la acci\'on evidente ``por la izquierda'' sobre
s\'i mismo: $g\in G$ act\'ua sobre los elementos $h\in H$
por multiplicaci\'on por la izquierda, i.e.,
\[g \mapsto (h\mapsto g\cdot h).\]
Est\'a tambien, claro est\'a, la acci\'on por la derecha
\[g \mapsto (h\mapsto h\cdot g^{-1}).\]
(Por qué es que $g\mapsto (h\mapsto h g)$ no es una acci\'on?)
Ninguna de estas dos acciones son interesantes
cuando se trata de aplicar directamente
el Lema \ref{lem:orbsta}, pues los estabilizadores
son triviales. Empero, tenemos tambi\'en la acci\'on {\em por conjugaci\'on}
\[g \mapsto (h\mapsto g h g^{-1}).\]
El estabilizador de un punto $h\in G$ no es sino su {\em centralizador}
$C(h)$, definido en (\ref{eq:pastrul});
la \'orbita de un punto $h\in G$ bajo la acci\'on de todo el grupo $G$
es la {\em clase de conjugaci\'on}
\[\Cl(h) = \{g h g^{-1}: g\in G\}.\]

As\'i, tenemos el siguiente resultado, crucial en lo que sigue. Su
importancia consiste en hacer que las cotas superiores (como las que
derivaremos m\'as tarde) sobre intersecciones con $\Cl(g)$
impliquen cotas inferiores sobre intersecciones con $C(g)$. 
La importancia de esto \'ultimo es que siempre es \'util saber que disponemos
de muchos elementos dentro de un {\em variedad} (tal como un toro).

\begin{lem}\label{lem:lawve}
  Sea $A\subset G$ un conjunto no vac\'io. Entonces,
  para todo $g\in A^l$, $l\geq 1$,
\[|A^{-1} A\cap C(g)|\geq \frac{|A|}{|A^{l+1} A^{-1}\cap \Cl(g)|}.\]
\end{lem}
\begin{proof}
  Sea $G\curvearrowright G$ la acci\'on de $G$ sobre s\'i mismo por conjugaci\'on. Aplique (\ref{eq:applepie}) con $x=g$;
  la \'orbita de $g$ bajo la acci\'on de  $A$ es un subconjunto de
  $A^{l+1} A^{-1}\cap \Cl(g)$.
\end{proof}

Es instructivo ver otras consecuencias de (\ref{eq:applepie}). La siguiente
nos muestra, por as\'i decirlo, que si obtenemos que la intersecci\'on
de $A$ con
un subgrupo $H$ de $G$ crezca, entonces hemos mostrado que $A$ mismo crece.
\begin{prob}
  Sea $G$ un grupo y $H$ un subgrupo de $G$. Sea $A\subset G$ un conjunto no
  vac\'io tal que
$A = A^{-1}$. Pruebe que, para todo $k>0$,
  \[|A^{k+1}|\geq \frac{|A^k\cap H|}{|A^2\cap H|} |A|.\]
(Sugerencia: considere la acci\'on $G\curvearrowright G/H$ por multiplicaci\'on por la izquierda, es decir, $g\mapsto (a H \mapsto g a H)$.)
\end{prob}


\chapter{Intersecciones con variedades}\label{chap:inter}
\section{Geometr\'ia algebraica extremadamente b\'asica}
\subsection{Variedades}
Una {\em variedad} (algebraica y af\'in)
en un espacio vectorial de $n$ dimensiones sobre
un cuerpo $K$ consiste en todos los puntos $(x_1,x_2,\dotsc,x_n)\in
\overline{K}^n$
que satisfacen un sistema de ecuaciones 
\begin{equation}\label{eq:qutu}
  P_i(x_1,\dotsc,x_n) = 0,\;\;\;\; 1\leq i\leq k,\end{equation}
donde $P_i$ son polinomios con coeficientes en $K$.

(Hay distintas maneras alternativas de formalizar el mismo concepto.
Podr\'iamos definir formalmente la variedad no exactamente
como el sistema de ecuaciones
en s\'i, sino como el conjunto de todas las ecuaciones polinomiales
implicadas por el sistema en (\ref{eq:qutu}). También hay definiciones de
apariencia mucho m\'as abstractas, basada en la teor\'ia de {\em esquemas}
(Grothendieck), pero no necesitaremos entrar all\'i.)

Se dice generalmente que los puntos $(x_1,\dotsc,x_n)$ que satisfacen
(\ref{eq:qutu}) {\em yacen} sobre la variedad, as\'i como hablamos de puntos
que yacen sobre una curva o superficie algebraica; claro est\'a, las curvas y las
superficies son casos especiales de variedades.
Dada una variedad $V$ definida sobre $K$ y un cuerpo $L$ tal que
$K\subset L \subset \overline{K}$, escribimos $V(L)$ por el conjunto
de todos los puntos  $(x_1,x_2,\dotsc,x_n)\in L^n$ que yacen sobre $V$.

El caso trivial es el de la variedad $\mathbb{A}^n$ ({\em espacio af\'in})
definida por el sistema vac\'io de ecuaciones (o por la ecuaci\'on 
$P(x_1,\dotsc,x_n)=0$, donde $P$ es el polinomio $0$). 
Claramente, $\mathbb{A}^n(L) = L^n$.

Dadas dos variedades $V_1$, $V_2$, tanto $V_1\cap V_2$ como
$V_1 \cup V_2$ son variedades: la variedad $V_1\cap V_2$ est\'a dada
por la uni\'on de las ecuaciones que definen $V_1$ y aquellas que definen
$V_2$, mientras que, si $V_1$ est\'a definida por (\ref{eq:qutu}) y $V_2$
est\'a definida por
\begin{equation}\label{eq:katu}
  Q_j(x_1,\dotsc,x_n) = 0,\;\;\;\; 1\leq j\leq k',\end{equation}
donde $Q_i$ son polinomios con coeficientes en $K$, entonces la variedad
$V_1\cup V_2$ est\'a dada por las ecuaciones
\[(P_i \cdot Q_j)(x_1,\dotsc,x_n) = 0,\;\;\;\; 1\leq i\leq k,
\;\;\;\; 1\leq i'\leq k'.\]

Consideremos ahora los grupos lineares, como $\SL_2(K)$. Est\'a claro que
$\SL_2(K)$ est\'a contenido en
\[M_2(\overline{K}) =
\left\{\left(\begin{matrix}x_1 & x_2\\x_3 & x_4\end{matrix}\right): x_1,x_2,x_3,x_4\in
  \overline{K}\right\},\]
  el cual es un espacio vectorial (de dimensi\'on $4$) sobre $\overline{K}$.
  Por lo tanto, tiene sentido hablar de variedades $V$
  en $M_2$. Por ejemplo, tenemos la variedad $V$
  de elementos que tienen una traza dada:
  \begin{equation}\label{eq:hut}
    \tr\left(\begin{matrix}x_1 & x_2\\x_3 & x_4\end{matrix}\right) = C,\;\;
    \text{i.e.,\;\;} x_1+x_4 - C = 0.\end{equation}
    
    Nuestro grupo $\SL_2$ es tambi\'en una variedad, dada por la
    ecuaci\'on $x_1 x_4 - x_2 x_3 = 1$. Es as\'i un ejemplo de un {\em
      grupo algebraico}. Estrictamente hablando, son los puntos
    $\SL_2(K)$ (o $\SL_2(L)$) del grupo algebraico $\SL_2$ los que
    forman un grupo en el sentido usual del t\'ermino. (Claro
    est\'a, la operaci\'on de grupo $\cdot:\SL_2\times \SL_2\to \SL_2$
    est\'a bien definida como un {\em morfismo de variedades afines} --
    una aplicaci\'on de una variedad a otra dado por polinomios.)

    Es f\'acil ver que un toro m\'aximo $T = C(g)$
    tambi\'en es un grupo algebraico, puesto que la ecuaci\'on
    \[h g = g h\]
    es un sistema de ecuaciones polinomiales (lineares, en verdad)
    sobre los coeficientes de $h\in \SL_2$. As\'i, $T$ es un subgrupo algebraico
    de $G$.
    
\subsection{Dimensi\'on, grado, intersecciones}
Dada una variedad $V$, una {\em subvariedad} $W$ es una variedad contenida en
\'el. Una variedad $V$ se dice irreducible si no es la uni\'on de dos
subvariedades no vac\'ias $W, W' \subsetneq V$.
Por ejemplo, $\SL_2$ es {\em irreducible}; \'esto es b\'asicamente
una consecuencia del hecho que el polinomio $x_1 x_4 - x_2 x_3 - 1$ es
irreducible.

     Podemos definir la {\em dimensi\'on} de un variedad
    irreducible $V$ como el entero $d$ m\'aximo tal que haya una cadena
    de variedades irreducibles no vac\'ias
    \[V_0\subsetneq V_1\subsetneq V_2 \subsetneq \dotsc \subsetneq V_d = V.\]
    \'Esto coincide con el concepto intuitivo de {\em dimensi\'on}:
    un plano contiene una l\'inea, que contiene a un
    punto, y as\'i un plano es de dimensi\'on por lo menos $2$; en verdad,
    es de dimensi\'on exactamente $2$.

    \noindent    {\em Hecho.} La dimensi\'on de $\mathbb{A}^n$ es $n$.

    En particular, la dimensi\'on de una variedad en $\mathbb{A}^n$ es
    siempre finita ($\leq n$).

        \begin{prob}
      Probemos que
      la intersecci\'on $\cap_{i\in I} V_i$ de una colecci\'on finita o infinita
      de variedades $V_i\in \mathbb{A}^n$, $i\in I$, es una variedad.

      Para una colecci\'on finita, \'esto es evidente. Por ende, para
      el caso infinito, bastar\'a si mostramos que hay un $S\subset I$
      finito tal que $\cap_{i\in I} V_i = \cap_{i\in S} V_i$.
    Muestre que esto se reduce a mostrar que una cadena de variedades
      \[\mathbb{A}^n \supsetneq W_1 \supsetneq W_2 \supsetneq W_3
      \supsetneq \dots\]
      debe ser finita. {\em (Esta propiedad se llama {\em propiedad Noetheriana}.)}

      Reduzca esto a su vez al caso de $W_1$ irreducible. Concluya
      por inducci\'on en la dimension de $W_1$.
    \end{prob}

\begin{prob}
  \begin{enumerate}
  \item Sea $V$ una variedad. Pruebe que
    se puede expresar $V$ como una uni\'on finita de variedades
irreducibles $W_1, W_2,\dotsc, W_k$. (Pista: use la propiedad Noetheriana.)
  \item   Sea $V$ irreducible. Muestre que,
    si $V$ es una uni\'on finita de variedades
      \[V_1, V_2,\dotsc, V_k,\] entonces
      existe un $1\leq i\leq k$ tal que $V_i = V$.
    \item Sea $V$ una variedad. Muestre que, si imponemos la condici\'on que
      $W_i \not\subset W_j$ para $i\ne j$, la decomposici\'on
      $V = W_1\cup W_2\cup \dotsc \cup W_k$ en variedades irreducibles
      $W_i$ es \'unica (excepto que,
      claro, los $W_i$ pueden permutarse).
      Las variedades $W_i$ se llaman {\em componentes irreducibles} de $V$.
  \end{enumerate}
\end{prob}

Si una variedad es una uni\'on de variedades irreducibles todas
    de dimensi\'on $d$, decimos que es ``de dimensi\'on pura'', y podemos
    decir que es de dimensi\'on $d$.

    Dadas dos variedades irreducibles
    $W\subset V$, la {\em codimensi\'on} $\codim(W)$ de $W$ en $V$ es
    simplemente $\dim(V)-\dim(W)$. Es f\'acil ver que, para $V$ irreducible,
    o bien $W=V$, o bien $\codim(W)>0$. Si $\codim(W)>0$ (o
    si $W$ es una uni\'on de variedades irreducibles de codimensi\'on
    positiva),
        podemos pensar en los 
        elementos de $W(\overline{K})$ como {\em especiales}, y en los
        elementos de $V(\overline{K})$ que no est\'an en $W(\overline{K})$
        como {\em gen\'ericos}.
        
 Es posible (y muy recomendable) considerar variedades m\'as generales
        que las variedades afines. Por ejemplo, podemos considerar
        las {\em variedades proyectivas}, definidas por sistemas de
        ecuaciones $P(x_0,x_1,\dotsc,x_n)=0$ donde cada $P$ es un 
polinomio homog\'eneo en $n+1$ variables.
 Los puntos
        en un variedad proyectiva viven en el {\em espacio proyectivo}
        $\mathbb{P}^{n}$. Los puntos del espacio proyectivo sobre un campo
        $L$ son elementos de
        $L^{n+1}$ (excepto $(0,0,\dotsc,0)$),
        donde se identifica dos elementos
        $x,x'\in L^{n+1}$ si $x$ es un m\'ultiplo
        escalar de $x'$, i.e., si $x = \lambda x'$
        para alg\'un $\lambda\in L$. En otras palabras,
        \[\mathbb{P}^n(L) =
        (L^{n+1}\setminus \{(0,0,\dotsc,0)\})/
        \sim,\;\;\; \text{donde $x\sim x'$ si $\exists \lambda\in
          L$ t.q. $x = \lambda x'$.}\]
        %

        La opci\'on de trabajar con variedades proyectivas nos da
        mucha libertad:
        en particular, es posible mostrar que podemos
        hablar de la variedad proyectiva de todas las
        l\'ineas (o todos los planos) en el espacio $n$-dimensional
        (proyectivo). El ejemplo m\'as simple es la variedad de todas
        las l\'ineas en el plano: como una l\'inea en el plano
        proyectivo ($n=2$) est\'a dada por una ecuaci\'on linear
        homog\'enea
        \[ c_0 x_0 + c_1 x_1 +  c_2 x_2 = 0,\]
        y como dos tales ecuaciones dan la misma l\'inea si
        sus triples $(c_1,c_2,c_3)$ son m\'ultiples el uno del otro,
        tenemos que las l\'ineas en el plano proyectivo est\'an en
        correspondencia uno-a-uno con $\mathbb{P}^2$ mismo.
        En $\mathbb{P}^n$, como en $\mathbb{A}^n$, podemos hablar
        de subvariedades, codimensi\'on, elementos gen\'ericos.
        Podemos hacer una inmers\'ion de $\mathbb{A}^n$ en $\mathbb{P}^n$:
        \[(x_1,x_2,\dotsc,x_n) \mapsto (1,x_1,x_2,\dotsc,x_n).\]
        El complemento es la subvariedad de $\mathbb{P}^n$ dada
        por $x_0=0$; para $n=2$, se le llama {\em recta en el infinito}.
        
        Contemplemos ahora una curva irreducible $C$ en el plano, es decir,
        una subvariedad de $\mathbb{A}^2$ de dimensi\'on $1$. (En verdad,
        no es necesaria la irreducibilidad; la supondremos realmente s\'olo
        al trabajar con an\'alogos en dimensiones superiores.)
        Consideremos
        tambi\'en una l\'inea $\ell$ en $\mathbb{A}^2$. Podr\'ia ser que
        $\ell$ fuera tangente a $C$, pero no es dif\'icil demostrar
        que tal es el caso s\'olo cuando $\ell$ yace en
        una subvariedad de $\mathbb{P}^2$ (la {\em curva dual a $C$}).
        En otras palabras, una l\'inea gen\'erica (se dice tambi\'en:
        {\em en posici\'on general}) no es tangente a $C$. El n\'umero
        de puntos de intersecci\'on en $\mathbb{P}^2(\overline{K})$
        de la curva $C$ con una l\'inea gen\'erica $\ell$ resulta ser
        independiente de $\ell$; llamamos a ese n\'umero el {\em grado}
        de $C$.

        Resulta ser que el grado de una curva irreducible
        en el plano dada por
        una ecuaci\'on \[P(x_1,x_2)=0\] (o por una ecuaci\'on $P(x_0,x_1,x_2)=0$,
        $P$ homogéneo) es simplemente el grado de $P$. La ventaja de la
        definici\'on que dimos del grado de una curva es que es m\'as
        conceptual y se generaliza de manera natural. En dimensiones
        superiores, la misma variedad puede ser definida por distintos
        sistemas de ecuaciones de grados distintos; deseamos una definici\'on
        de {\em grado} que dependa s\'olo de la variedad, y no del sistema
        que la define.

        Si $V$ es una variedad de dimensi\'on $2$ en $\mathbb{A}^3$,
        definimos su grado como su n\'umero de intersecciones con
        una recta gen\'erica; si $V$ es de dimensi\'on $1$
        en $\mathbb{A}^3$, definimos
        su grado como su n\'umero de intersecciones
        con un {\em plano} gen\'erico. As\'i como hablamos de l\'ineas
        y planos, podemos definir, en general,
        una {\em variedad linear} mediante
        ecuaciones lineares.
        Para $V$ una variedad irreducible en $\mathbb{A}^n$
        de dimensi\'on $d$, definimos el {\em grado} $\deg(V)$ de $V$
        como el n\'umero de intersecciones de $V$ con una variedad linear
gen\'erica
        de codimensi\'on $d$ en $\mathbb{A}^n$.
        La misma definici\'on
        es v\'alida cuando $V$ es no necesariamente irreducible pero de
        dimensi\'on pura e igual a $d$.

        Si bien este grado, como dec\'iamos, no tiene porque corresponder
        al grado de ninguna de las ecuaciones en un sistema de ecuaciones
        que defina a $V$, puede ser acotado por una constante
        que depende s\'olo de los grados de tales ecuaciones y su n\'umero.
        Ese es un caso especial de lo que estamos por discutir.

        El {\em teorema de B\'ezout}, en el plano, nos dice que, para dos
        curvas irreducibles distintas $C_1$, $C_2$ en $\mathbb{A}^2$, el
        n\'umero de puntos de la intersecci\'on $(C_1\cap C_2)(\overline{K})$
        es a lo m\'as $d_1 d_2$. (En verdad, si consideramos $C_1$ y $C_2$
        gen\'ericos, o trabajamos en $\mathbb{P}^2$ y contamos ``multiplicidades'', el n\'umero de puntos de intersecci\'on es {\em exactamente} $d_1 d_2$.)

        En general, si $V_1$ y $V_2$ son variedades irreducibles, y escribimos
        $V_1\cap V_2$ como una uni\'on de variedades irreducibles
        $W_1, W_2,\dotsc, W_k$, con $W_i\not\subset W_j$ para $i\ne j$,
        una generalizaci\'on del teorema de B\'ezout nos dice que
        \begin{equation}\label{eq:dur}
          \sum_{i=1}^k \deg(W_k) \leq \deg(V_1) \deg(V_2).\end{equation}
        (V\'ease, por ejemplo, \cite[p.251]{MR1658464},
        donde se menciona a Fulton y
        \foreignlanguage{english}{MacPherson} en conexi\'on a
        (\ref{eq:dur}) y enunciados m\'as generales.)
        Toda variedad irreducible de dimensi\'on $0$ consiste en un
        \'unico punto; por ello (\ref{eq:dur}) implica el teorema de
        B\'ezout habitual.
        
        \section{Escape de subvariedades}
        
        Sea $G$ un grupo que actua por transformaciones lineares
        sobre el espacio $n$-dimensional $K^n$, $K$ un cuerpo.
        (En otras palabras, se nos es dado un homomorfismo
        $\phi:G\to \GL_n(K)$ de $G$ al
        grupo de matrices invertibles $\GL_n(K)$.)
        Sea $W$ una variedad de codimensi\'on positiva en $\mathbb{A}^n$.
        Est\'abamos llamando a los elementos de $W(\overline{K})$
        {\em especiales}, y a los otros elementos de
        $\mathbb{A}^n(\overline{K})$ {\em gen\'ericos}.
        
        Sea $A$ un conjunto de generadores de $G$ y $x$ un punto de $W$.
        Muy bien podr\'ia
        ser que la \'orbita $A\cdot x$ est\'e contenida por entero en $W$.
        Empero, como veremos ahora, si $G x$ no est\'a contenida en $W$,
        entonces
        siempre es posible {\em escapar} de $W$ en un n\'umero acotado
        de pasos: no s\'olo que habr\'a (por definici\'on) alg\'un producto
$g$ de un n\'umero finito de elementos de $A$ y $A^{-1}$ 
tal que $g\cdot x$ est\'a fuera de $W$, sino que
        habr\'a un producto (a decir verdad, muchos productos)
        $g\in (A\cup A^{-1})^k$, $k$ {\em acotado},
        tal que $g\cdot x$ est\'a fuera de $W$. En otras palabras, si
        escapamos por lo menos una vez, eventualmente, de $W$, escapamos
        de muchas maneras de $W$, despu\'es de un n\'umero acotado $k$ de
        pasos.

        La prueba\footnote{El enunciado de la proposici\'on es c\'omo en \cite{HeSL3},
            basado en \cite{MR2129706}, pero la idea es probablemente
            m\'as antigua.} procede por inducci\'on en la dimensi\'on, controlando
        el grado.

        \begin{prop}\label{prop:huru}
          Sean dados:
          \begin{itemize}
            \item $G$, un grupo actuando por transformaciones lineares sobre
              $K^n$, $K$ un cuerpo;
            \item $W\subsetneq \mathbb{A}^n$, una variedad,
                      \item un conjunto de generadores $A\subset G$;
        \item un elemento $x\in \mathbb{A}^n(K)$ tal que $G \cdot x$ no est\'a contenido
          en $W$.
          \end{itemize}
          
          Entonces hay constantes $k$, $c$ que dependen s\'olo del n\'umero,
          dimensi\'on y grado de los componentes irreducibles de $W$, tales que
          hay por lo menos $\max(1,c |A|)$ elementos $g\in (A \cup A^{-1} \cup
          \{e\})^k$ tales que $g x\notin W(K)$.
        \end{prop}
        Para aclarar el proceso de inducci\'on, daremos primero la
        prueba en un caso particular. Una variedad {\em linear} es simplemente
        una l\'inea, un espacio, etc.; en otras palabras, es una variedad
        definida por ecuaciones lineares.
        \begin{proof}[Prueba para $W$ linear e irreducible]
          Sea $W$ linear e irreducible. Podemos asumir sin p\'erdida
          de generalidad que $A = A^{-1}$ y $e\in A$.
          
          Procederemos por inducci\'on
          en la dimensi\'on de $W$. Si $\dim(W)=0$,
          entonces $W$ consiste en un s\'olo punto $x_0$, y el enunciado que
          queremos probar
          es cierto: existe un $g\in A$ tal que
          $g x \ne x_0$ (por qu\'e?); si hay menos de $|A|/2$ tales
          elementos, escogemos un $g_0\in A$ tal que $g_0 x_0 \ne x_0$
          (por qu\'e existe?),
          y entonces, para cada uno de los m\'as de $|A|/2$ elementos
          $g\in A$ tales que $g x = x_0$, tenemos que
          $g_0 g x = g_0 x_0 \ne x_0$.

          Asumamos, entonces, que $\dim(W)>0$, y que el enunciado ha sido
          probado para todas las variedades lineares irreducibles $W'$
          con $\dim(W')<\dim(W)$. Si $g W = W$ para todo $g\in A$,
          entonces ya sea (a) $g x\not\in W(K)$ para todo $g\in A$,
          y el enunciado es inmediato,
          o (b) $g x\in W(K)$
          para todo $g\in G$ (puesto que $G$ est\'a generado por $A$), lo
          cual est\'a en contradicci\'on con nuestras suposiciones. Podemos
          asumir, entonces, que $g W \ne W$ para algun $g\in A$.

          Entonces $W' = g W \cap W$ es una variedad linear irreducible de
          dimensi\'on $\dim(W')<\dim(W)$. Por lo tanto, por la hip\'otesis
          inductiva, hay $\geq \max(1, c' |A|)$ elementos $g'$ de
          $A^{k'}$ (donde $c'$ y $k'$ dependen s\'olo de $\dim(W)$)
          tales que $g' x$ no yace en la variedad $W' = g W \cap W$.
          Entonces, para cada tal $g'$, ya sea $g^{-1} g' x$ o $g' x$ no
          yace en $W$. As\'i, hemos probado la proposicion con
          $c = c'/2$, $k = k'+1$.
          \end{proof}

        \begin{prob}
          Generalize la prueba que acabamos de dar de tal manera que
          d\'e Prop.~\ref{prop:huru} para $W$ arbitrario. {\em Sugerencia:}
          como un primer paso, generalice la prueba de tal manera que
          funcione para toda uni\'on $W$ de variedades lineares irreducibles.
          (\'Esto ya exigir\'a adaptar el proceso de inducci\'on de tal
          manera que se controle de alguna manera
          el n\'umero de componentes en cada paso. Claro est\'a, la
          intersecci\'on de dos uniones $W$ de $d$ variedades lineares
          irreducibles tiene a lo m\'as $d^2$ componentes.)
          Luego muestre que la prueba es v\'alida para toda union $W$
          de variedades irreducibles, no necesariamente lineares,
          utilizando la generalizaci\'on (\ref{eq:dur}) del teorema
          de B\'ezout.
          \end{prob}

\section{Estimaciones dimensionales}
Dado un conjunto de generadores $A\subset \SL_2(K)$ (o $A\subset \SL_n(K)$, o
lo que se desee) y una subvariedad $V$ de codimensi\'on positiva en
$\SL_2$, sabemos que una proporci\'on positiva de los elementos de
$A^k$, $k$ acotado, yacen fuera de $V$: \'este es un caso especial de la
Proposici\'on \ref{prop:huru} (con $x$ igual a la identidad $e$).

Aunque esto desde ya implica una
cota superior
para el n\'umero de elementos de $A^k$ en $V(K)$, podemos dar una cota
mucho mejor. Los estimados de este tipo pueden trazarse en parte a
\cite{LP} (caso de $A$ un subgrupo, $V$ general) y en parte a \cite{Hel08}
y \cite{HeSL3} ($A$ un conjunto en general, pero $V$ especial).
Tales cotas tienen en general la forma
\begin{equation}\label{eq:utur}
  |A\cap V(K)|\ll |(A\cup A^{-1} \cup \{e\})^k|^{\frac{\dim V}{\dim G}}.\end{equation}
Se lograron cotas completamente generales del tipo (\ref{eq:utur}) en
\cite{BGT} y \cite{PS} ($A$ y $V$ arbitrarios, $G$ un grupo linear algebraico
simple, como en \cite{LP}). 

Como primero paso hacia la estrategia general, veamos un caso particular
de manera muy concreta (aunque no lo usemos al final).
La prueba es b\'asicamente la misma que en
 \cite[\S 4]{Hel08}.

\begin{lem}\label{lem:rokto}
  Sea $G=\SL_2$, $K$ un cuerpo, y $T$ un toro m\'aximo.
  Sea $A\subset G(K)$ un conjunto de generadores de $G(K)$. 
  Entonces
  \begin{equation}\label{eq:clodo}
    |A\cap T(K)|\ll |(A \cup A^{-1}\cup \{e\})^k|^{1/3}\end{equation}
  donde $k$ y la constante impl\'icita son constantes absolutas.
\end{lem}
\begin{proof}
  Podemos suponer sin p\'erdida de generalidad que
  $|K|$ es mayor que una constante, pues, de lo contrario, la conclusi\'on
  es trivial. Tambi\'en
  podemos suponer sin p\'erdida de generalidad que $A = A^{-1}$, $e\in A$,
  y que $|A|$ es mayor que una constante, reemplazando $A$ por
  $(A \cup A^{-1}\cup \{e\})^c$, $c$ constante, de ser necesario. 
    Podemos tambi\'en escribir los elementos de $T$ como
  matrices diagonales, conjugando por un elemento de
  $\SL_2(\overline{K})$.

  Sea \begin{equation}\label{eq:murut}
    g=\left(\begin{matrix} a & b\\ c & d\end{matrix}\right)\end{equation} un elemento
  cualquiera de $\SL_2(\overline{K})$ con $a b c d\ne 0$. Consideremos la aplicaci\'on
  $\phi:T(K) \times T(K) \times T(K) \to G(K)$ dada por
  \[\phi(x, y, z) = x \cdot g y g^{-1} \cdot z.\]
  Queremos mostrar que esta aplicaci\'on es en alg\'un sentido cercana
  a ser inyectiva. (La raz\'on de tal estrategia? Si la aplicaci\'on fuera
  inyectiva, y tuvieramos $g\in A^\ell$, $\ell$ una constante,
  entonces tendr\'iamos
 \[|A\cap T(K)|^3 = |\phi(A\cap T(K),A\cap T(K),A\cap T(K))| 
  \leq |A A^\ell A A^{-\ell} A| = |A^{2\ell+3}|,\]
  lo cual implicar\'ia inmediatamente el resultado que queremos.
Simplemente estamos usando el hecho que el tama\~no de la imagen $\phi(D)$
de una inyecci\'on $\phi$ tiene el mismo n\'umero de elementos que
su dominio $D$.)

Multiplicando matrices, vemos que, para 
\[x=\left(\begin{matrix} r & 0 \\0 & r^{-1}\end{matrix}\right),\;\;
y = 
\left(\begin{matrix} s & 0 \\0 & s^{-1}\end{matrix}\right),\;\;
z = 
  \left(\begin{matrix} t & 0 \\0 & t^{-1}\end{matrix}\right),\]
$\phi((x,y,z))$  es igual a
\begin{equation}\label{eq:malvot}  \left(\begin{matrix} r t (s a d - s^{-1} b c) &
    r t^{-1} (s^{-1} - s) a b\\
    r^{-1} t (s - s^{-1}) c d &
r^{-1} t^{-1} (s^{-1} a d - s b c)\end{matrix}\right).\end{equation}
Sea $s\in \overline{K}$ tal que $s^{-1} - s \ne 0$ y
$s a d - s^{-1} b c \ne 0$.
Un breve c\'alculo muestra que entonces
$\phi^{-1}(\{\phi((x,y,z))\})$ tiene a lo m\'as $16$ elementos: tenemos que
\[  r t^{-1} (s^{-1} - s) a b \cdot 
    r^{-1} t (s - s^{-1}) c d  = - (s - s^{-1})^2 a b c d,\]
    y, como $a b c d \ne 0$, hay a lo m\'as $4$ valores de $s$
    dado un valor de $- (s-s^{-1})^2 a b c d$ (el producto de las esquinas
    superior derecha e inferior izquierda de (\ref{eq:malvot})); para cada tal valor de $s$,
    el producto y el cociente de las esquinas superior izquierda
    y superior derecha de (\ref{eq:malvot}) determinan $r^2$ y $t^2$,
    respectivamente, y obviamente hay s\'olo $2$ valores de $r$ y
    $2$ valores de $t$ para $r^2$ y $t^2$ dados.

    Ahora bien, hay a lo m\'as $4$ valores de $s$ tales que
    $s^{-1} - s = 0$ o $s a d - s^{-1} b c = 0$.
    Por lo tanto, tenemos que
    \[|\phi(A\cap T(K),A\cap T(K),A\cap T(K))| \geq
    \frac{1}{16} |A\cap T(K)| (|A\cap T(K)| - 4) |A\cap T(K)|,\]
    y, como antes,
    $\phi(A\cap T(K),A\cap T(K),A\cap T(K))\subset
    A A^\ell A A^{-\ell} A = A^{2\ell+3}$.
    Si $|A\cap T(K)|$
    es menor que $8$ (o cualquier
    otra constante) entonces la conclusi\'on (\ref{eq:clodo}) es trivial.
    Por lo tanto, conclu\'imos que
    \[|A\cap T(K)|^3 \leq 2 |A\cap T(K)| (|A\cap T(K)| - 4) |A\cap T(K)|
    \leq 32 |A^{2\ell+3}|,\]
    i.e., (\ref{eq:clodo}) es cierta.

    S\'olo queda verificar que existe un elemento (\ref{eq:murut})
    de $A^\ell$ con $a b c d \ne 0$. Ahora bien, $a b c d= 0$
    define una subvariedad $W$ de $\mathbb{A}^4 \sim M_2$;
    m\'as a\'un, para $|K|>2$,
    existen elementos de $G(K)$
    fuera de tal variedad. Por lo tanto, las condiciones de
    Prop.~\ref{prop:huru} se cumplen (con $x$ igual a la identidad $e$).
    As\'i, obtenemos que existe $g\in A^\ell$ ($\ell$ una constante)
    tal que $g\not\in W(K)$, lo cual era lo que necesit\'abamos.
\end{proof}

Hagamos abstracci\'on de lo que acabamos de hacer, para as\'i poder
generalizar el resultado a una variedad arbitraria $V$ en vez de $T$.
Trataremos el caso de $V$ de dimensi\'on $1$, por conveniencia.
La estrategia de la prueba del Lema \ref{lem:rokto} consiste en construir
un morfismo $\phi:V\times V\times \dotsb \times V\to G$
($r$ copias de $V$, donde $r = \dim(G)$) de la forma
\begin{equation}\label{eq:kukuku}
  \phi(v_1,\dotsc,v_{r}) = v_1 g_1 v_2 g_2 \dotsb v_{r-1} g_{r-1} v_r,\end{equation}
donde $g_1,g_2,\dotsc,g_{r-1}\in A^\ell$, y mostrar 
que, para $v = (v_1,\dotsc,v_{r})$ gen\'erico (es decir, fuera de una subvariedad de $V\times \dotsb \times V$
de codimensi\'on positiva), la preimagen
$\phi^{-1}(\phi(v))$ tiene dimensi\'on $0$. En verdad, como acabamos de ver,
es suficiente mostrar que esto es cierto para $(g_1,g_2,\dotsc,g_{r-1})$
un elemento gen\'erico de $G^{r-1}$; el argumento de escape (Prop.~\ref{prop:huru}) se encarga del
resto.

Para hacer que el argumento marche para $V$ general (y $G$ general),
es necesario asumir algunos fundamentos. Esencialmente, tenemos la
elecci\'on de ya sea trabajar sobre el \'algebra de tipo Lie
o introducir un poco m\'as de geometr\'ia algebraica.
La primera elecci\'on (tomada en \cite{MR3348442}, siguiendo a
\cite{HeSL3}) asume que el lector tiene cierta familiaridad con los
grupos y \'algebras de Lie, y que sabe, o est\'a dispuesto a creer,
que la relaci\'on entre grupos y \'algebras de Lie sigue la misma
si trabajamos sobre un cuerpo finito en vez de $\mathbb{R}$ o
$\mathbb{C}$. La segunda elecci\'on -- que tomaremos aqu\'i --
requiere saber, o estar listo a
aceptar, un par de hechos b\'asicos sobre morfismos, v\'alidos sobre
cuerpos arbitrarios.

No importa gran cosa si se sigue el uno u el otro formalismo.
Los fundamentos, en uno y otro caso, se sentaron s\'olidamente en la
primera mitad del siglo XX (Zariski, Chevalley, etc.) y son
relativamente accesibles. Los lectores que sientan inter\'es en estudiar
las bases del camino que seguiremos est\'an invitados a leer
\cite[Ch. 1]{MR1748380} (de por s\'i una excelente idea) o cualquier
texto similar.

 Est\'a claro que, si $\phi:\mathbb{A}^n\to \mathbb{A}^m$ es un morfismo
 y $V\subset \mathbb{A}^m$ es una variedad, entonces la preimagen
 $\phi^{-1}(V)$ es una variedad (por qu\'e?). Algo nada evidente que
 utilizaremos es el hecho que, si $\phi$ es como dijimos y
 $V\subset \mathbb{A}^n$ es una variedad,
 entonces  $\phi(V)$ es un {\em conjunto constru\'ible}, lo cual quiere decir
 una uni\'on finita de t\'erminos de la forma $W\setminus W'$,
 donde $W$ y $W'\subset W$ son variedades. (Por ejemplo, si $V\subset \mathbb{A}^2$
 es la variedad dada por $x_1 x_2 = 1$ (una hip\'erbola), entonces
 su imagen bajo el morfismo $\phi(x_1,x_2) = x_1$ es el conjunto
 constru\'ible $\mathbb{A}^1 \setminus \{0\}$.)
 \'Este es un teorema de Chevalley \cite[\S I.8, Cor.~2]{MR1748380};
 encapsula parte del
campo cl\'asico llamado {\em teor\'ia de la eliminaci\'on}.
Es f\'acil deducir, que, en general, para $V$ constru\'ible, $\phi(V)$
es constru\'ible.

Siempre podemos expresar un conjunto constru\'ible $S$ como una uni\'on
$\cup_i (W_i\setminus W'_i)$ con $\dim(W'_i)<\dim(W_i)$. (Por qu\'e?)
La {\em clausura de Zariski} $\overline{S}$
del conjunto constru\'ible $S$ es entonces
$\cup_i W_i$.

El siguiente lema es \cite[Prop. 1.5.30]{MR3309986},
lo cual es a su vez en esencia \cite[Lemma 4.5]{LP}.
(Murmullo de un mundo paralelo: en el formalismo que no seguimos, esto
corresponde al hecho, b\'asico pero no trivial, que el \'algebra de un
grupo de tipo Lie simple es simple.)

Decimos que un grupo algebraico $G$ es {\em casi simple}
si no tiene ning\'un subgrupo algebraico normal $H$ de dimensi\'on
positiva y menor que $\dim(G)$.
(Por ejemplo, $\SL_n$ es casi simple para todo $n\geq 2$.)
\begin{lem}\label{lem:tush}
  Sea $G\subset \SL_n$ un grupo algebraico irreducible y
  casi simple definido sobre
  un cuerpo $K$.
  Sean $V', V\subsetneq G$ subvariedades con $\dim(V')>0$.
  Entonces, para todo
  $g\in G(\overline{K})$ fuera de una subvariedad $W\subsetneq G$,
  alg\'un componente de la clausura de Zariski $\overline{V' g V}$
  tiene dimensi\'on $>\dim(V)$.

  M\'as a\'un, el n\'umero de componentes
  de $W$ y sus grados est\'an acotados por una constante que depende s\'olo
  de $n$ y del n\'umero y grados de los componentes de $V'$ y $V$.
\end{lem}
\begin{proof}

  Podemos suponer sin p\'erdida de generalidad que $V$ y $V'$ son
  irreducibles, y que $e\in V'(\overline{K})$.

  Sea $g\in G(\overline{K})$. Supongamos que $\overline{V' g V}$ tiene
  dimensi\'on $\leq \dim(V)$. Para todo $v'\in V'(\overline{K})$,
  $v' g V$ es una variedad de dimensi\'on $\dim(V)$ -- uno del n\'umero
  finito de componentes $W_i$ de $\overline{V' g V}$ de dimensi\'on
  $\dim(V)$. Para cada tal $W_i$, los puntos $v'$ tales que
  $v' g V = W_i$ forman una variedad $V_i$, al ser la intersecci\'on
  de variedades \[\bigcap_{v\in V(\overline{K})} W_i v^{-1} g^{-1}.\]
  As\'i, $V'$ es la uni\'on de un n\'umero finito de variedades $V_i$;
  como $V'$ es irreducible, esto implica que $V'=V_i$ para alg\'un $V$.
  En particular, $g V = e\cdot g V = W_i$. Por lo tanto, $V' g V = g V$.

  Ahora bien, los elementos $g\in G(\overline{K})$ tales que
  $V' g V = g V$ son una intersecci\'on
  \[\mathop{\bigcap_{v\in V}}_{v'\in V'} \phi_{v',v}^{-1}(V)\]
  de variedades $\phi_{v',v}^{-1}(V)$, donde $\phi_{v',v}(g) = g^{-1} v' g v$.
  Por lo tanto, tales elementos constituyen una
  subvariedad $W$ de $G$; m\'as a\'un, gracias a B\'ezout (\ref{eq:dur}),
  su n\'umero de componentes
  as\'i como el grado de estos est\'an acotados por
  una constante que depende s\'olo
  de $n$ y del n\'umero y grados de los componentes de $V'$ y $V$.

  Falta s\'olo mostrar que $W\ne G$. Supongamos que $W=G$. Entonces
  $V' g V = g V$ para todo $g\in G(\overline{K})$. Ahora bien,
  el estabilizador $\{g\in G(\overline{K}): g V = V\}$ de $V$
  no s\'olo es un grupo, sino que es (el conjunto de puntos de) una
  variedad (nuevamente: para ver esto,
  expr\'eselo como una intersecci\'on de variedades).
  Llamemos a tal variedad $\Stab(V)$. Tenemos que $g^{-1} V' g \subset
  \Stab(V)$ para todo $g\in G(\overline{K})$, y por lo tanto
  \[V' \subset \bigcap_{g\in G(\overline{K})} g \Stab(V) g^{-1}.\]
Esto muestra que la variedad $\cap_{g\in G(\overline{K})} g \Stab(V) g^{-1}$
  es de dimensi\'on $\geq \dim(V')>0$. Al mismo tiempo,
  dicha variedad es un subgrupo algebraico normal de $G$, contenido en
  $\Stab(V)$. Como $\Stab(V)\subsetneq G$, tenemos
  un subgrupo algebraico normal de $G$, de dimensi\'on positiva y
  estrictamente contenido en $G$. En otras palabras, $G$ no es un grupo
  algebraico casi simple.
  Contradicci\'on.
\end{proof}

Sean $G$, $V$ y $V'$ tales que satisfagan las hip\'otesis del
Lema \ref{lem:tush}, y sea $g\in G(\overline{K})$ como en la conclusi\'on
del Lema, i.e., fuera de la subvariedad $W\subsetneq G$.
Asumamos que $\dim(V') = 1$, y consideremos el morfismo
$\phi:V'\times V\to \overline{V'gV}$ dado por
\[\phi(v',v) = v' g v.\]
La dimension de la imagen de un morfismo no es mayor que la dimension
de su dominio (ejercicio), as\'i que
\[\dim(\phi(V' \times V))\leq \dim(V'\times V) = \dim(V') + \dim(V) = \dim(V)+1.\]
Al mismo tiempo, por el Lema \ref{lem:tush}, $\dim(\phi(V',V))>\dim(V)$.
Por lo tanto, \[\phi(V' \times V) = \dim(V)+1 = \dim(V'\times V).\]

Si un morfismo $\phi:X\to X'$ es tal que $\overline{\phi(X)} = X'$, decimos
que $\phi$ es {\em dominante}. (Por ejemplo, el morfismo $\phi$
que acabamos de considerar es dominante.) Aceptemos el hecho que, si
$\phi$ es dominante y $\dim(X')=\dim(X)$, entonces hay una subvariedad
$Y\subsetneq X$ tal que, para todo $x\in X(\overline{K})$ que no yazca
en $Y(\overline{K})$, la variedad $\phi^{-1}(\phi(\{x\}))$
tiene dimensi\'on $0$. (\'Esta es una consecuencia inmediata del
\cite[\S 1.8, Thm. 3]{MR1748380}.) M\'as a\'un, el n\'umero de componentes
de $Y$ y su grado est\'an acotados en t\'erminos del grado de $\phi$
y del n\'umero, dimensi\'on y grado de los componentes de $X$ y $X'$.

(Para que lo que acabamos de llamar una consecuencia inmediata de algo
en otra parte se vuelva intuitivamente claro, considere el caso
$K = \mathbb{R}$. Entonces
$Y$ es la subvariedad de $X$ definida por la condici\'on ``la determinante
$D\phi(x)$ de $\phi$ en el punto $x$ tiene determinante $0$''.
Hay varias maneras de ver que $Y$ es una subvariedad $\subsetneq X$
para $K$ arbitrario: como dijimos,
es posible definir derivadas sobre cuerpos arbitrarios, o, alternativamente,
proceder como en \cite[\S 1.8, Thm. 3]{MR1748380}.)

Aplicando esto a la aplicaci\'on $\phi$ que ten\'iamos, obtenemos
que hay una subvariedad $Y\subsetneq V'\times V$ tal que, para todo
$x\in (V'\times V)(\overline{K})$ que no yace en $Y$, $\phi^{-1}(\phi(x))$
tiene dimensi\'on $0$.

Dado \'esto, podemos probar una generalizaci\'on del Lema \ref{lem:rokto}.
Se trata realmente de (\ref{eq:utur}) para toda variedad de dimensi\'on
$1$.
\begin{prop}\label{prop:cheyen}
  Sea $K$ un cuerpo y $G\subset \SL_n$ un grupo algebraico casi simple
  tal que $|G(K)|\geq c |K|^{\dim(G)}$, $c>0$.
  Sea $Z\subset G$ una variedad de dimensi\'on $1$.
  Sea $A\subset G(K)$ un conjunto de generadores de $G(K)$.
  Entonces
    \begin{equation}\label{eq:clada}
    |A\cap Z(K)|\ll |(A \cup A^{-1}\cup \{e\})^k|^{1/\dim(G)},\end{equation}
    donde $k$ y la constante impl\'icita dependen solamente de $n$, de $c$ y del
    n\'umero y grado de los componentes irreducibles de $G$ y $Z$.
\end{prop}
Obviamente, $G = \SL_n$ es una elecci\'on v\'alida, pues es casi simple
y $|\SL_n(K)|\gg |K|^{n^2-1} = |K|^{\dim(G)}$.
\begin{prob}
  Pruebe la Proposici\'on \ref{prop:cheyen}. He aqu\'i un esbozo:
  \begin{enumerate}
  \item Muestre el siguiente lema b\'asico:
    si $W\subset \mathbb{A}^N$ es una variedad de dimensi\'on $d$, entonces
    el n\'umero de puntos $(x_1,\dotsc,x_N)\in \mathbb{A}^n(K)$
    que yacen en $W$
    es $\ll |K|^d$, donde la constante impl\'icita depende s\'olo de $N$
    y del n\'umero y grado de componentes irreducibles de $W$. (Sugerencia:
    para $d=0$, \'esto est\'a claro. Para $d>0$, considere la proyecci\'on
    $\pi:\mathbb{A}^N\to\mathbb{A}^{N-1}$ a las primeras $N-1$ coordenadas,
    o m\'as bien dicho la restricci\'on $\pi|W$ de $\pi$ a $W$. Reduzca al
    caso de dimensi\'on $d-1$ -- la manera de hacerlo depende de si
    $\pi|W$ es o no es dominante.)
  \item Utilizando el escape de subvariedades (Prop.~\ref{prop:huru})
    y el Lema \ref{lem:tush},
    muestre que, dadas las condiciones del Lema \ref{lem:tush},
    existe un elemento $g\in (A\cup A^{-1} \cup \{e\})^\ell$,
    $\ell$ una constante (dependiendo de esto y aquello),
    tal que  alg\'un componente de la clausura de Zariski $\overline{V' g V}$
    tiene dimensi\'on $>\dim(V)$. Esto es rutina, pero no olvide mostrar
    que hay algun punto de $G(K)$ fuera de $W$ (usando el lema
    b\'asico que acaba de probar).
  \item Aplicando esto (y las consecuencias discutidas inmediatamente
    despues del Lema \ref{lem:tush}) de manera iterada, muestre que
    existen $g_1,\dotsc,g_{r-1} \in (A\cup A^{-1} \cup \{e\})^{\ell'}$
    y una subvariedad $Y\subsetneq Z\times Z \times \dotsc \times Z$
    ($r=\dim(G)$ veces) tales que, para
    todo $x\in (Z\times Z \times \dotsc \times Z)(\overline{K})$
    que no yace en $Y$, $\phi^{-1}(\phi(x))$ es de dimensi\'on $0$,
    donde $\phi$ es como en (\ref{eq:kukuku}).
  \item Usando nuevamente un argumento que distingue si una
    projecci\'on (esta vez de $Z\times \dotsc \times Z$ ($r$ veces)
    a $Z\times \dotsc \times Z$ ($r-1$ veces)) es dominante,
    e iterando, muestre que hay a lo m\'as $O(|A\cap Z(K)|^{r-1})$
    elementos de $(A\cup Z(K)) \times \dotsc \times (A\cup Z(K))$
    ($r$ veces) en $Y$.
  \item Concluya que la proposici\'on Prop.~\ref{prop:cheyen} es
    cierta.
        \end{enumerate}
\end{prob}

En general, se puede probar (\ref{eq:utur}) para $\dim(V)$ arbitrario
siguiendo argumentos muy
similares, mezclados con una inducci\'on sobre la dimensi\'on de la variedad
$V$ en (\ref{eq:utur}). Ilustraremos el proceso b\'asico haciendo
las cosas en detalle para $G=\SL_2$ y para
el tipo de variedad $V$ que realmente necesitamos.

Se trata de la variedad 
$V_t$ definida por
\begin{equation}\label{eq:schlingue}
  \det(g) = 1, \tr(g) = t
    \end{equation}
  para $t\ne \pm 2$. Tales variedades nos interesan por el hecho que,
  para cualquier $g\in \SL_2(K)$ regular semisimple (lo cual en $\SL_2$
  quiere decir: con dos valores propios distintos), la clase de conjugacion
  $\Cl(g)$ est\'a contenida en $V_{\tr(g)}$.


  \begin{prop}\label{prop:juru}
    Sea $K$ un cuerpo; sea $A\subset \SL_2(K)$ un conjunto de generadores de
    $\SL_2(K)$.
    Sea $V_t$ dada por (\ref{eq:schlingue}). Entonces, para todo $t\in K$ aparte de $\pm 2$,
    \begin{equation}\label{eq:terka}
      |A\cap V_t(K)|\ll |(A\cup A^{-1} \cup \{e\})^k|^{\frac{2}{3}},
      \end{equation}
   donde $k$ y la constante impl\'icita son constantes absolutas. 
  \end{prop}
 Claro est\'a, $\dim(\SL_2) = 3$ y $\dim(V_t) = 2$, as\'i que este es un
  caso particular de (\ref{eq:utur}).

  \begin{proof}
    Consideremos la aplicaci\'on $\phi:V_t(K)\times V_t(K) \to
    \SL_2(K) $ definida por
    \[\phi(y_1,y_2) = y_1 y_2^{-1}.\]
    Est\'a claro que \[\phi(A\cap V_t(K),A\cap V_t(K))\subset A^2.\]
    As\'i, si $\phi$ fuera inyectiva, tendr\'iamos inmediatamente que
    $|A\cap V_t(K)|^2\leq |A^2|$. Ahora bien, $\phi$ no es
    inyectiva. La preimagen de $\{h\}$, $h\in \SL_2(K)$, es
    \[\phi^{-1}(\{h\}) = \{(w,h^{-1} w): \tr(w) = t, tr(h^{-1} w) = t\}.\]

    Debemos preguntarnos, entonces, cu\'antos elementos de $A$ yacen
    en la subvariedad $Z_{t,h}$ de $G$ definida por
    \[Z_{t,h} =  \{(w,h w): \tr(w) = t, tr(h^{-1} w) = t\}.\]
    Para $h\ne \pm e$, $\dim(Z_{t,h}) = 1$ (verificar), y el n\'umero
    y grado de componentes de $Z_{t,h}$ esta acotado por una constante
    absoluta. As\'i, aplicando la Proposici\'on \ref{prop:cheyen},
    obtenemos que, para $h\ne \pm e$,
    \[|A\cap Z_{t,h}(K)|\ll |A^{k'}|^{1/3},\]
    donde $k'$ y la constante impl\'icita son absolutas.

    Ahora bien, para cada $y_1\in V_t(K)$, hay por lo menos $|V_t(K)|-2$
    elementos $y_2\in V_t(K)$ tales que $y_1 y_2^{-1} \ne \pm e$.
    Conclu\'imos que
    \[|A\cap V(K)| (|A\cap V(K)|-2) \leq |A^2|\cdot \max_{g\ne \pm e}
    |A\cap Z_{t,h}(K)| \ll |A^2| |A^{k'}|^{1/3}.\]
    Podemos asumir que $|A\cap V(K)|\geq 3$, pues de lo contrario
    la conclusi\'on deseada es trivial. Obtenemos, entonces, que
    \[|A\cap V(K)|\ll |A^k|^{2/3}\]
    para $k = \max(2,k')$, como quer\'iamos.
      \end{proof}
  

  Pasemos a la consecuencia que nos interesa.
  \begin{coro}\label{cor:hutz}
    Sea $K$ un cuerpo y $G = \SL_2$. Sea $A$ un conjunto de generadores de
    $G(K)$; sea $g\in A^\ell$ ($\ell\geq 1$) regular semisimple. Entonces
    \begin{equation}\label{eq:hop1}
      |A^{-1} A\cap C(g)|\gg \frac{|A|}{|(A\cup A^{-1}\cup \{e\})^{k \ell}|^{2/3}},\end{equation}
    donde $k$ y  la constante impl\'icita son absolutas.

    En particular, si $|A^3|\leq |A|^{1+\delta}$, entonces
    \begin{equation}\label{eq:hop2}
      |A^{-1} A \cap C(g)|\gg_\ell |A|^{1/3-O(\delta \ell)}.\end{equation}
  \end{coro}
  \begin{proof}
    La Proposici\'on \ref{prop:juru} y el Lema \ref{lem:lawve}
    implican (\ref{eq:hop1}) inmediatamente, puesto que
$\Cl(g)\subset V_{\tr(g)}$, donde $V_t$ se define como en
(\ref{eq:schlingue}). De manera tambi\'en muy sencilla,
    la conclusi\'on (\ref{eq:hop2}) se deduce de
    (\ref{eq:hop1}) a trav\'es de (\ref{eq:marmundo}).
    \end{proof}

  Veamos ahora dos problemas cuyos resultados no utilizaremos;
  son esenciales, empero, si se quiere trabajar en $\SL_n$ para $n$
  arbitrario. El primer problema es relativamente ambicioso, pero
  ya hemos visto todos los elementos esenciales para su soluci\'on.
  En esencia, s\'olo se trata de saber organizar la recursi\'on.

  \begin{prob}
    Generalice la Proposicion \ref{prop:cheyen} a $Z$ de dimensi\'on
    arbitraria.
  \end{prob}
  
  En general, un elemento $g\in \SL_n(K)$ es {\em regular semisimple}
  si tiene $n$ valores propios distintos. Claro est\'a, todo elemento
  de $C(g)$ tiene los mismos vectores propios que $g$. Cuando $G=\SL_n$, como
  para $\SL_2$, los elementos de $C(g)$ son los puntos $T(K)$ de un
  subgrupo algebraico abeliano $T$ de $G$, llamado un toro m\'aximo.
  Tenemos que $\dim(T) = n-1$ y $\dim(\overline{\Cl(g)})=\dim(G) - \dim(T)$.

  \begin{prob}
    Generalice \ref{cor:hutz} a $G=\SL_n$, para $g$ semisimple.
    En vez de (\ref{eq:hop2}),
    la conclusi\'on reza como sigue:
        \begin{equation}\label{eq:hosop2}
      |A^{-1} A\cap C(g)|\gg |A|^{\frac{\dim(T)}{\dim(G)}-O(\delta)},\end{equation}
      donde las constantes impl\'icitas dependen solo de $n$.
  \end{prob}

  Terminemos por una breve nota con un lado anecd\'otico. Una versi\'on
  del Corolario \ref{cor:hutz} fue probada en \cite{Hel08}, donde jug\'o
  un rol central. Luego fue generalizada a $\SL_n$ en \cite{HeSL3}, dando,
  en esencia, (\ref{eq:hosop2}).

  Empero, estas versiones tenian una debilidad: daban (\ref{eq:hop2})
  y (\ref{eq:hosop2}) para la mayor\'ia de los
  $g\in A^\ell$, y no para {\em todo} $g\in A^\ell$. Esto hac\'ia que el resto del argumento -- la parte que estamos
  por ver -- fuera m\'as complicado y dif\'icil de generalizar que lo
  es hoy en d\'ia.

  La moraleja es, por supuesto, que no hay que asumir  que las t\'ecnicas
  y argumentos que a uno le son familiares son \'optimos -- y que para
  simplificar una prueba vale la pena tratar de
  probar resultados intermedios m\'as fuertes.
  
  \chapter{El crecimiento en $\SL_2(K)$}
  \section{El caso de los subconjuntos grandes}\label{sec:subgra}
  Veamos primero que pasa con $A\cdot A\cdot A$ cuando
  $A\subset \SL_2(\mathbb{F}_q)$ es grande con respecto a $G=\SL_2(\mathbb{F}_q)$. En verdad no es dif\'icil
  mostrar que, si $|A| \geq |G|^{1-\delta}$, $\delta>0$ suficientemente
  peque\~no, entonces
  $(A \cup A^{-1} \cup \{e\})^k = G$, donde $k$ es una constante absoluta.
  Probaremos algo m\'as fuerte: $A^3=G$. La prueba se 
  debe a Nikolov y Pyber \cite{MR2800484}; est\'a basada sobre
  una idea cl\'asica, desarrollada en este contexto por Gowers \cite{MR2410393}.
  Nos dar\'a la oportunidad de revisitar el tema de los valores propios
  de la matriz de adyacencia
  $\mathscr{A}$ de $\Gamma(G,A)$. (Los comenzamos a discutir
   en \S \ref{sec:dfsq}.)

   Primero, recordemos que una {\em representaci\'on compleja} de un grupo
   $G$ es un homomorphismo $\phi:G\to \GL_d(\mathbb{C})$; decimos, naturalmente,
   que $d\geq 1$ es la {\em dimensi\'on} de la representaci\'on. Una
   representaci\'on $\phi$ es {\em trivial} si $\phi(g)=e$ para todo $g\in G$.

   El siguiente resultado se debe a Frobenius (1896), por lo menos para $q$
   primo. Se puede mostrar
   simplemente examinando una tabla de caracteres, como en
   \cite{MR1691549} (que da tambi\'en an\'alogos de esta proposici\'on para otros grupos de tipo Lie). Para $q$ primo,
   hay una prueba breve y elegante; v\'ease, e.g.,
   \cite[Lemma 1.3.3]{MR3309986}.
   \begin{prop}
     Sea $G = \SL_2(\mathbb{F}_q)$, $q = p^\alpha$.
     Entonces toda representaci\'on compleja no trivial de $G$ tiene
     dimensi\'on $\geq (q-1)/2$.
   \end{prop}

   Ahora bien, para cada valor propio $\nu$ de $\mathscr{A}$, podemos
   considerar su {\em espacio propio} -- el espacio vectorial que consiste
   en todas las funciones propias $f:G\to \mathbb{C}$ con valor propio
   $\nu$. Como puede verse de la definici\'on de $\mathscr{A}$
   (inmediatamente despu\'es de (\ref{eq:dudur})), tal espacio es invariante
   bajo la acci\'on de $G$ por multiplicaci\'on por la derecha. En otras
   palabras, es una representaci\'on de $G$ - y puede ser trivial s\'olo
   si se trata del espacio (uni-dimensional) que consiste de las funciones 
   constantes, i.e., el espacio propio que corresponde al valor propio
   $\nu_0=1$.
   Por lo tanto, todo los otros valores propios tienen multiplicidad
   $\geq (q-1)/2$.  Asumamos, como es nuestra
   costumbre, que $A = A^{-1}$, lo cual implica que todos los valores
   propios son reales:
   \[ \dotsc \leq \nu_2 \leq \nu_1 \leq \nu_0 = 1.\]

   La idea es ahora es obtener un hueco espectral, i.e.,
   una cota superior para $\nu_j$, $j>0$.
   Es muy com\'un usar el hecho
   que la traza de una potencia $\mathscr{A}^r$ de
   una matriz de adyacencia puede expresarse de dos maneras:
   como el n\'umero (normalizado por el factor $1/|A|^r$, en nuestro caso)
   de ciclos de longitud $r$ en el grafo $\Gamma(G,A)$, por una parte,
   y como la suma de potencias $r$-\'esimas de los valores propios de
   $\mathscr{A}$, por otra. En nuestro caso, para $r=2$, esto nos da
   \begin{equation}\label{eq:gotra}
     \frac{|G| |A|}{|A|^2} = \sum_j \nu_j^2 \geq \frac{q-1}{2}
   \nu_j^2,\end{equation}
   para cualquier $j\geq 1$, y, por lo tanto,
   \begin{equation}\label{eq:himult}
|\nu_j|\leq \sqrt{\frac{|G|/|A|}{(q-1)/2}}.\end{equation}
   \'Esta es una cota superior muy peque\~na
   para $|A|$ grande. Esto
   quiere decir que unas cuantas aplicaciones de $\mathscr{A}$ bastan
   para hacer que una funci\'on se ``uniformice'', i.e., se vuelva
   casi constante, pues cualquier componente
   ortogonal al espacio propio de funciones constantes es multiplicado por
   alg\'un
   $\nu_j$, $j\geq 1$, en cada paso. La prueba siguiente simplemente aplica
   esta observaci\'on.
   
   \begin{prop}[\cite{MR2800484}]\label{prop:diplo}
Sea $G = \SL_2(\mathbb{F}_q)$, $q=p^\alpha$. Sea $A\subset G$,
$A = A^{-1}$. Asumamos $|A|\geq 2 |G|^{8/9}$. Entonces
\[A^3 = G.\]
   \end{prop}
   La suposicion $A = A^{-1}$ es en verdad innecesaria, gracias al trabajo
   adicional puesto en \cite{MR2410393} para el caso no sim\'etrico.
\begin{proof}
Supongamos $g\in G$ such that $g\notin A^3$. Entonces el producto escalar
\[\langle \mathscr{A} 1_A, 1_{g A}\rangle =
\sum_{x\in G} (\mathscr{A} 1_A)(x) \cdot 1_{g A^{-1}}(x)\]
es igual a $0$.
Podemos asumir que los vectores propios $v_j$ satisfacen $\langle v_j,v_j\rangle = 1$. Entonces
\[\begin{aligned}
\langle \mathscr{A} 1_A, 1_{g A}\rangle &=
\langle \sum_{j\geq 0} \nu_j \langle 1_A, v_j\rangle v_j, 1_{g A}\rangle
\\ &= \nu_0 \langle 1_A,v_0\rangle \langle v_0, 1_{g A^{-1}}\rangle +
\sum_{j>0} \nu_j \langle 1_A, v_j\rangle \langle v_j, 1_{g A^{-1}}\rangle
.\end{aligned}\] Ahora bien, $v_0$ es la funci\'on constante, y, al
satisfacer  $\langle v_0,v_0\rangle=1$, es igual
a $1/\sqrt{|G|}$. Luego
\[\nu_0 \langle 1_A,v_0\rangle \langle v_0, 1_{g A^{-1}}\rangle =
1\cdot \frac{|A|}{\sqrt{|G|}}\cdot \frac{|g^{-1} A|}{\sqrt{|G|}} = 
\frac{|A|^2}{|G|}.\]
Al mismo tiempo, gracias a (\ref{eq:himult}) y
Cauchy-Schwarz,
\[\begin{aligned} \left|\sum_{j>0} \nu_j \langle 1_A, v_j\rangle \langle v_j, 1_{g A^{-1}}\rangle\right|
&\leq
\sqrt{\frac{2 |G|/|A|}{q-1}}
\sqrt{\sum_{j\geq 1}  |\langle 1_A, v_j\rangle|^2}
\sqrt{\sum_{j\geq 1} |\langle v_j, 1_{g A^{-1}}\rangle|^2}\\
&\leq 
\sqrt{\frac{2 |G|/|A|}{q-1}} |1_A|_2 |1_{g A^{-1}}|_2
=  \sqrt{\frac{2|G| |A|}{q-1}} .
\end{aligned}\]
Como $|G|= (q^2-q)q$, tenemos que $|A|\geq 2 |G|^{8/9}$ implica que
\[\frac{|A|^2}{|G|} > \sqrt{\frac{2 |G| |A|}{q-1}},\] y por lo tanto
$\langle \mathscr{A} 1_A, 1_{g A^{-1}}\rangle$ es mayor que $0$. Contradicci\'on.
\end{proof}

  \section{El crecimiento en $\SL_2(K)$, $K$ arbitrario}
  Probemos finalmente el teorema \ref{thm:main08}. En esta parte nos acercaremos
  m\'as a tratamientos nuevos (en particular, \cite{PS})
  que al tratamiento original en \cite{Hel08}; estos tratamientos nuevos
  se generalizan m\'as f\'acilmente. Si bien s\'olo deseamos presentar
  una prueba para $\SL_2$, notaremos el punto o dos en la prueba
  d\'onde hay que trabajar un poco a la hora de generalizarla para $\SL_n$.

  La primera prueba de este teorema en la literatura utilizaba el
 {\em teorema de la suma y producto}, un resultado no trivial de combinatoria
 aditiva. La prueba que daremos no lo utiliza, pero s\'i tiene algo en com\'un
 con su prueba: la inducci\'on, usada de una manera particular. En esencia,
 si algo es cierto para el paso $n$, pero no para el paso $n+1$, se trata
 de usar ese mismo hecho para obtener la conclusi\'on que deseamos
 de otra manera (lo que se llama un ``fulcro'' ({\em pivot})
 en la prueba que estamos por
 ver).
 El hecho que estemos en un grupo sin un orden natural ($n$, $n+1$, etc.)
 resulta ser irrelevante.

 \begin{proof}[Prueba del Teorema \ref{thm:main08}]
   Gracias a (\ref{eq:mony}), podemos asumir que $A=A^{-1}$ y $e\in A$.
   Tambi\'en podemos asumir que $|A|$ es mayor que una constante absoluta,
   pues de lo contrario la conclusi\'on es trivial. Escribamos $G=\SL_2$.

   Supongamos que $|A^3|<|A|^{1+\delta}$, donde $\delta>0$ es una peque\~na
   constante a ser determinada m\'as tarde. Por escape (Prop.~\ref{prop:huru}),
   existe un elemento $g_0\in A^c$ regular semisimple
   (esto es, $\tr(g_0)\ne \pm 2$), donde $c$ es una constante absoluta.
   (A decir verdad, $c=2$; ejercicio opcional.)
   Su centralizador en $G(K)$
   es $C(g) = T(\overline{K})\cap G(K)$ para alg\'un toro
   maximal $T$.

   Llamemos a $\xi\in G(K)$ un {\em fulcro} si la funci\'on $\phi_g:A\times C(g)
   \to G(K)$ definida por
   \begin{equation}\label{eq:naksym}
     (a,t) \mapsto a \xi t \xi^{-1}
   \end{equation}
   es inyectiva en tanto que funci\'on de $\pm e\cdot
   A/\{\pm e\} \times C(g)/\{\pm e\}$
   a $G(K)/\{\pm e\}$.

   {\em   Caso (a): Hay un fulcro $\xi$ en $A$.} Por el Corolario
   \ref{cor:hutz}, existen $\gg |A|^{1/3 - O (c\delta)}$ elementos de $C(g)$
   en $A^2$. Por lo tanto, por la inyectividad de $\phi_\xi$,
   \[\left|\phi_\xi(A,A^2\cap C(g))\right| \geq \frac{1}{4} |A| |A^2\cap C(g)|
   \gg |A|^{\frac{4}{3} - O(c \delta)}.\]
   Al mismo tiempo, $\phi_\xi(A,A^2\cap C(g))\subset A^5$, y por lo tanto
   \[|A^5|\gg |A|^{4/3 - O(c \delta)}.\]
   Para $|A|$ mayor que una constante y $\delta>0$ menor que una constante,
   esto nos da una contradicci\'on con $|A^3| < |A|^{1+\delta}$ (por Ruzsa
   (\ref{eq:jotor})).

   {\em Caso (b): No hay fulcros $\xi$ en $G(K)$.} Entonces, para todo
   $\xi\in G(K)$, hay $a_1,a_2\in A$, $t_1,t_2\in T(K)$,
   $(a_1,t_1) \ne (\pm a_2, \pm t_2)$ tales que $a_1 \xi t_1 \xi^{-1} =
   \pm e \cdot a_2 \xi t_2 \xi^{-1}$, lo cual da
   \[a_2^{-1} a_1 = \pm e\cdot \xi t_2 t_1^{-1} \xi^{-1}.\]
   En otras palabras, para cada $\xi\in G(K)$, $A^{-1} A$ tiene una
   intersecci\'on no trivial con el toro $\xi T \xi^{-1}$:
   \begin{equation}\label{eq:nortsch}
     A^{-1} A \cap \xi T(K) \xi^{-1} \ne \{\pm e\}.
   \end{equation}
   (Por cierto, esto s\'olo es posible si $K$ es un cuerpo finito
   $\mathbb{F}_q$. Por qu\'e?)
   
   Escoja cualquier $g\in A^{-1} A \cap \xi T(K) \xi^{-1}$ con $g\ne \pm e$.
   Entonces $g$ es regular semisimple (nota: esto es peculiar a $\SL_2$)
   y su centralizador $C(g)$ es igual a $\xi T(K) \xi^{-1}$ (por qu\'e?).
   Por lo tanto, por el corolario \ref{cor:hutz}, obtenemos que hay
   $\geq c' |A|^{1/3 - O(\delta)}$ elementos de $\xi T(K) \xi^{-1}$ en $A^2$,
   donde $c'$ y la constante impl\'icita son absolutas.

   Por lo menos $(1/2) |G(K)|/|T(K)|$ toros m\'aximos de $G$ son de la
   forma $\xi T \xi^{-1}$, $\xi \in G(K)$ (demostrar!). Tambi\'en tenemos
   que todo elemento de $G$ que no sea $\pm e$ puede estar en a lo m\'as
   un toro m\'aximo (de nuevo algo peculiar a $\SL_2$). Por lo tanto,
   \[|A^2|\geq \frac{1}{2} \frac{|G(K)|}{|T(K)|} (c' |A|^{1/3 - O(\delta)} - 2)
   \gg q^2 |A|^{1/3 - O(\delta)}.\]

   Por lo tanto, ya sea $|A^2|> |A|^{1+\delta}$ (en contradicci\'on con
   la suposici\'on que $|A^3|\leq |A|^{1+\delta}$)  o
   $|A|\geq |G|^{1-O(\delta)}$. En el segundo caso,
la proposici\'on \ref{prop:diplo} implica que $A^3 = G$.   

{\em Caso (c): Hay elementos de $G(K)$ que son fulcros y otros que no lo son.}
Como $\langle A\rangle = G(K)$, esto implica que existe un $\xi\in G$
que no es un fulcro y un $a\in A$ tal que $a\xi \in G$ s\'i es un fulcro.
Como $\xi$ no es un fulcro, (\ref{eq:nortsch}) es cierto, y por lo tanto
hay $|A|^{1/3-O(\delta)}$ elementos de $\xi T \xi^{-1}$ en $A^k$.

  Al mismo tiempo, $a\xi$ es un fulcro, i.e., la aplicaci\'on
  $\phi_{a\xi}$ definida en (\ref{eq:naksym}) es inyectiva
  (considerada como una aplicaci\'on de
$A/\{\pm e\} \times C(g)/\{\pm e\}$
   a $G(K)/\{\pm e\}$).
  Por lo tanto,
  \[\left|\phi_{a \xi}(A, \xi^{-1} (A^k \cap \xi T \xi^{-1}) \xi)\right| \geq \frac{1}{4}
|A| |A^k \cap \xi T \xi^{-1}| \geq \frac{1}{4} |A|^{\frac{4}{3} - O(\delta)}.\]
Como $\phi_{a \xi}(A, \xi^{-1} (A^k \cap \xi T \xi^{-1}) \xi) \subset A^{k+3}$,
obtenemos que \begin{equation}\label{eq:matameri}
|A^{k+3}|\geq \frac{1}{4} |A|^{4/3 - O(\delta)}.\end{equation}
Gracias otra vez a Ruzsa (\ref{eq:jotor}), esto contradice
$|A^3|\leq |A|^{1+\delta}$ para $\delta$ suficientemente peque\~no.
 \end{proof}



 \appendix
 
 \chapter{Expansi\'on en $\SL_2(\mathbb{Z}/p\mathbb{Z})$}\label{chap:appa}

 Daremos aqu\'i un esbozo de c\'omo Bourgain y Gamburd probaron
 que, para $A_0\subset \SL_2(\mathbb{Z})$ tal que
 $\langle A_0\rangle$ es Zariski-denso, entonces
   \[\{\Gamma(\SL_2(\mathbb{Z}/p\mathbb{Z}),A_0 \mo p)\}_{\text{$p>C$, $p$ primo}}\]
   es una familia de expansores, i.e., tiene un hueco espectral constante.

   Primero, clarifiquemos que quiere decir ``Zariski-denso''. Esto quiere
   decir simplemente que no existe ninguna subvariedad $V\subsetneq \SL_2(\mathbb{C})$ que contenga a $\langle A_0\rangle$. Como dijimos en (\ref{sec:res}),
   es un hecho conocido que esto implica que $A_0 \mod p$ genera
   $\SL_2(\mathbb{Z}/p\mathbb{Z})$ para $p$ mayor que una constante $C$
   (\cite{MR763908}, \cite{MR735226}, \cite{MR880952} y
\cite{MR1329903} lo prueban para $\SL_2$ y para muchos grupos m\'as).

Es sencillo pasar a un subgrupo libre:
\[\Gamma(2) = \{g\in \SL_2(\mathbb{Z}): g\equiv I \mod 2\}\]
es libre, y es un resultado est\'andar (Nielsen-Schreier) que todo subgrupo
de un grupo libre es libre.
Por lo tanto $\langle A_0\rangle\cap \Gamma(2)$ es libre. El \'indice
de $\langle A_0\rangle\cap \Gamma(2)$
en $\langle A_0\rangle$ es
finito (por qu\'e?), y podemos encontrar un conjunto finito
que genera $\langle A_0\rangle\cap \Gamma(2)$ (generadores de Schreier, por
ejemplo). Esto es suficiente para que podamos asumir, sin p\'erdida de
generalidad, que $\langle A_0\rangle$ es libre.

Lo que ahora haremos es considerar la funci\'on
\[\mu(x) = \begin{cases} \frac{1}{|A_0 \mo p|} &\text{si $x\in A_0 \mo p$,}\\
0 &\text{si $x\not\in A_0 \mo p$}\end{cases}\]
y sus convoluciones. 
La convoluci\'on $f\cdot g$ de dos funciones
$f,g:G\to \mathbb{C}$ se define por
\[(f\cdot g)(x) = \sum_{y\in G} f(x y^{-1}) g(y).\]
La norma $\ell_p$ de una funci\'on $f:G\to \mathbb{C}$ es
\[|f|_p = \left(\sum_{y\in G} |f(y)|^p\right)^{1/p}.\]
Es f\'acil ver que la convoluci\'on
$\mu^{(\ell)} := \mu \ast \mu \ast \dotsc \ast \mu$ ($\ell$ veces) tiene norma
$\ell_1$ igual a $1$. Empero, la norma $\ell_2$ var\'ia. Para todo $f$,
$|f\ast \mu|_2\leq |f|_2$, por Cauchy-Schwarz,
con igualdad s\'olo si $f$ es uniforme, es decir, constante (ejercicio).
Por lo tanto, $|\mu^{(\ell)}|_2$ decrece cuando $\ell$ aumenta.

Nos interesa saber que tan r\'apido decrece, pues \'esto nos da
informaci\'on sobre los valores propios de $\mathscr{A}$. Veamos
por qu\'e. El operador
$\mathscr{A}$ no es sino la convoluci\'on por $\mu$. Podemos comparar, como
en (\ref{eq:gotra}), dos expresiones para la traza. Por una parte, la
traza de $\mathscr{A}^{2 \ell}$ es igual a la suma, para todo $g$,
del n\'umero de maneras de ir de $g$ a $g$ tomando productos por $A$
exactamente $2\ell$ veces, dividido por $|A_0|^{2\ell}$; esto es
\[|G| \mu^{(2 \ell)}(e) = |G| \sum_{x\in G} \mu^{(\ell)}(x^{-1}) \mu^{(\ell)}(x) =
|G| |\mu^{(\ell)}|_2^2,\]
donde $G=\SL_2(\mathbb{Z}/p\mathbb{Z})$.
(Como de costumbre, asumimos que $A_0 = A_0^{-1}$.)
Por otra parte, la traza de $\mathscr{A}^{2\ell}$ es igual a
$\sum_i \nu_i^{2\ell}$, donde $1=\nu_0> \nu_1 \geq \dotsc$ son los
valores propios de $\mathscr{A}$.

Como ya discutimos en \S \ref{sec:subgra}, todo valor propio
$\nu_j$, $j\geq 1$, tiene multiplicidad $\geq (p-1)/2$. Por lo tanto,
tenemos que, para todo $j\geq 1$,
\[\frac{p-1}{2} \nu_j^{2\ell} \leq \sum_{j\geq 0} \nu_j^{2\ell}
= |G| |\mu^{(\ell)}|_2^2.\]
Nuestra meta ser\'a mostrar que, para alg\'un $\ell\leq C \log p$,
$C$ una constante suficientemente grande, la funci\'on $\mu^{(\ell)}$ es razonablemente
uniforme, o ``llana'', por lo menos del punto de vista de su norma $\ell_2$:
$|\mu^{(\ell)}|_2^2 \ll 1/|G|^{1-\epsilon}$. (La distribuci\'on uniforme
tiene norma $\ell_2$ igual a $1/|G|$, naturalmente.) Entonces tendremos
\[\nu_j^{2\ell} \ll \frac{|G|^\epsilon}{p} \ll \frac{1}{p^{1-3\epsilon}}\]
(puesto que $|G|\ll p^3$) y por lo tanto
\[\nu_j \leq e^{- \frac{(1 - 3 \epsilon) \log p}{C \log p}}\leq 1 - \delta,\] 
donde $\delta>0$ es una constante (que, como $C$, puede depender de $A_0$).
Esto es lo que deseamos. 

(El uso de la multiplicidad de $\nu_j$
en este contexto particular remonta a Sarnak-Xue \cite{SarnakXue}.)

Lo que queda es, como dec\'iamos, mostrar que $|\mu^{(\ell)}|$ decrece
r\'apidamente cuando $\ell$ aumenta. \'Esto esta estrechamente ligado
a mostrar que $|A_0^\ell|$ decrece (en particular, lo implica), pero
no es trivialmente equivalente.

La prueba tiene dos pasos. Primero, igual que para $|A_0^\ell|$ (ver el
ejercicio \ref{ej:luk} y los comentarios que lo siguen), est\'a el caso de
lo que pasa para $\ell \leq \epsilon' \log p$, donde $\epsilon'$ es lo
suficientemente peque\~no como para que, para elementos
$g_1,g_2,\dotsc,g_{2\ell} \in A_0 = A_0 \cup A_0^{-1}$ cualesquiera, tengamos
que ninguno de los coeficientes de la matriz $g_1 g_2 \dotsc g_{2\ell}
\in \SL_2(\mathbb{Z})$ tenga valor absoluto $\geq p-1$. Entonces,
tenemos que no existen $x_i\in A_0 \mo p$, $1\leq i\leq k$,
$x_{i+1}\notin \{x_i,x_i^{-1}\}$ para
$1\leq i\leq k-1$, $x_i\ne e$ para $1\leq i\leq k$,
y $r_i\in \mathbb{Z}$, $r_i\ne 0$, $\sum_{1\leq i\leq k}
|r_i|\leq 2 \ell$,
tales que
\[x_1^{r_1} \dotsb x_k^{r_k} = e.\]
(Idea: si un elemento de $\SL_2(\mathbb{Z})$
es congruente $\mo p$ a la identidad sin ser la
identidad: entonces por lo menos uno de sus coeficientes de matriz tiene
valor absoluto por lo menos $p-1$.)

Esto implica inmediatamente que
los productos
de elementos de $A_0 \mo p$ de longitud $\ell$ son todos diferentes
(excepto por las igualdades obvias del tipo
$x\cdot e = x$ y $x\cdot x^{-1} = e$).
Por lo tanto, 
$|(A\mo p)^\ell|$ crece exponencialmente:
$|(A_0\mo p)^\ell| \geq (|A_0|-2)^\ell$.
\'Esta era
la parte crucial a soluci\'on del ejercicio \ref{ej:luk}.
Mostrar que $|\mu^\ell|_2$
decrece tambi\'en exponencialmente no es mucho m\'as dif\'icil, sobre
todo porque podemos asumir que $A_0$ es m\'as grande que una constante.
(Para $A_0$ m\'as peque\~no que una constante, ser\'ia un asunto m\'as
delicado: se trata
de un resultado cl\'asico de Kesten \cite{MR0109367} sobre los grupos libres.)

Queda por ver como decrece $|\mu^\ell|_2$ para $\epsilon' \log p \leq \ell 
\leq C \log p$. Aqu\'i que Bourgain y Gamburd muestran que, si tuvieramos
\[|\mu^{2\ell}|_2 > |\mu^{\ell}|_2^{1+\delta'},\]
$\delta'>0$, entonces existe un conjunto $A'\subset \SL_2(\mathbb{Z}/p
\mathbb{Z})$
tal que $|A'^3|<|A'|^{1+O(\delta')}$.
(La herramienta principal es 
el teorema de Balog-Szemer\'edi \cite{MR1305895}, fortalecido por Gowers
\cite{MR1844079} y generalizado
por Tao \cite{MR2501249} al caso no conmutativo.)
Muestran tambi\'en que $\mu(A')$ es grande, por lo cual $A'$ es menor
que $|G|^{1-O(\delta')}$ a menos que $\mu$ ya sea tan uniforme como deseamos.
Un argumento auxiliar muestra que $A'$ genera $\SL_2(\mathbb{Z}/p\mathbb{Z})$.
Por lo tanto, $|A'^3|<|A'|^{1+O(\delta')}$ entra
en contradicci\'on con el Teorema \ref{thm:main08}.

Esto muestra que $|\mu^{2\ell}|_2 \leq |\mu^{\ell}|_2^{1+\delta'}$
para $\epsilon' \log p \leq \ell \leq C \log p$, y termina la prueba.
Concluimos que $\nu_1\leq 1 - \delta$, que era lo que quer\'iamos demostrar.

\bibliographystyle{alpha}
\bibliography{cusco}

\def\cprime{$'$}
\begin{thebibliography}{{Moo}96}

\bibitem[BBS04]{BBS04}
L.~Babai, R.~Beals, and {\'A}.~Seress.
\newblock On the diameter of the symmetric group: polynomial bounds.
\newblock In {\em Proceedings of the {F}ifteenth {A}nnual {ACM}-{SIAM}
  {S}ymposium on {D}iscrete {A}lgorithms}, pages 1108--1112 (electronic), New
  York, 2004. ACM.

\bibitem[BG08]{MR2415383}
J.~Bourgain and A.~Gamburd.
\newblock Uniform expansion bounds for {C}ayley graphs of {${\rm
  SL}_2(\mathbb{F}_p)$}.
\newblock {\em Ann. of Math. (2)}, 167(2):625--642, 2008.

\bibitem[BGT11]{BGT}
E.~Breuillard, B.~Green, and T.~Tao.
\newblock Approximate subgroups of linear groups.
\newblock {\em Geom. Funct. Anal.}, 21(4):774--819, 2011.

\bibitem[BGT12]{BGTstru}
E.~Breuillard, B.~Green, and T.~Tao.
\newblock The structure of approximate groups.
\newblock {\em Publications math\'ematiques de l'IH\'ES}, 116:115--221, 2012.

\bibitem[BS88]{BS88}
L.~Babai and {\'A}.~Seress.
\newblock On the diameter of {C}ayley graphs of the symmetric group.
\newblock {\em J. Combin. Theory Ser. A}, 49(1):175--179, 1988.

\bibitem[BS94]{MR1305895}
A.~Balog and E.~Szemer{\'e}di.
\newblock A statistical theorem of set addition.
\newblock {\em Combinatorica}, 14(3):263--268, 1994.

\bibitem[Con]{Conrad}
K.~Conrad.
\newblock Simplicity of {$\PSL_n(F)$}.
\newblock
  \url{http://www.math.uconn.edu/~kconrad/blurbs/grouptheory/PSLnsimple.pdf}.

\bibitem[Din11]{MR2788087}
O.~Dinai.
\newblock Growth in {${\rm SL}_2$} over finite fields.
\newblock {\em J. Group Theory}, 14(2):273--297, 2011.

\bibitem[DS98]{MR1658464}
V.~I. Danilov and V.~V. Shokurov.
\newblock {\em Algebraic curves, algebraic manifolds and schemes}.
\newblock Springer-Verlag, Berlin, 1998.
\newblock Translated from the 1988 Russian original by D. Coray and V. N.
  Shokurov, Translation edited and with an introduction by I. R. Shafarevich,
  Reprint of the original English edition from the series Encyclopaedia of
  Mathematical Sciences [{\em Algebraic geometry. I}, Encyclopaedia Math. Sci.,
  23, Springer, Berlin, 1994; MR1287418 (95b:14001)].

\bibitem[DSC93]{MR1245303}
P.~Diaconis and L.~Saloff-Coste.
\newblock Comparison techniques for random walk on finite groups.
\newblock {\em Ann. Probab.}, 21(4):2131--2156, 1993.

\bibitem[EMO05]{MR2129706}
A.~Eskin, Sh. Mozes, and H.~Oh.
\newblock On uniform exponential growth for linear groups.
\newblock {\em Invent. math.}, 160(1):1--30, 2005.

\bibitem[Fre73]{MR0360496}
G.~A. Fre{\u\i}man.
\newblock {\em Foundations of a structural theory of set addition}.
\newblock American Mathematical Society, Providence, R. I., 1973.
\newblock Translated from the Russian, Translations of Mathematical Monographs,
  Vol 37.

\bibitem[GH11]{GH1}
N.~Gill and H.~A. Helfgott.
\newblock Growth of small generating sets in {${\rm SL}_n(\Bbb Z/p\Bbb Z)$}.
\newblock {\em Int. Math. Res. Not. IMRN}, (18):4226--4251, 2011.

\bibitem[Gow01]{MR1844079}
W.~T. Gowers.
\newblock A new proof of {S}zemer\'edi's theorem.
\newblock {\em Geom. Funct. Anal.}, 11(3):465--588, 2001.

\bibitem[Gow08]{MR2410393}
W.~T. Gowers.
\newblock Quasirandom groups.
\newblock {\em Combin. Probab. Comput.}, 17(3):363--387, 2008.

\bibitem[Gro81]{MR623534}
M.~Gromov.
\newblock Groups of polynomial growth and expanding maps.
\newblock {\em Inst. Hautes {\'E}tudes Sci. Publ. Math.}, (53):53--73, 1981.

\bibitem[Hel08]{Hel08}
H.~A. Helfgott.
\newblock Growth and generation in {${\rm SL}_2(\mathbb{Z}/p\mathbb{Z})$}.
\newblock {\em Ann. of Math. (2)}, 167(2):601--623, 2008.

\bibitem[Hel11]{HeSL3}
H.~A. Helfgott.
\newblock Growth in {${\rm SL}_3(\mathbb{Z}/p\mathbb{Z})$}.
\newblock {\em J. Eur. Math. Soc. (JEMS)}, 13(3):761--851, 2011.

\bibitem[Hel15]{MR3348442}
H.~A. Helfgott.
\newblock Growth in groups: ideas and perspectives.
\newblock {\em Bull. Amer. Math. Soc. (N.S.)}, 52(3):357--413, 2015.

\bibitem[HP95]{MR1329903}
E.~Hrushovski and A.~Pillay.
\newblock Definable subgroups of algebraic groups over finite fields.
\newblock {\em J. Reine Angew. Math.}, 462:69--91, 1995.

\bibitem[Hru12]{MR2833482}
E.~Hrushovski.
\newblock Stable group theory and approximate subgroups.
\newblock {\em J. Amer. Math. Soc.}, 25(1):189--243, 2012.

\bibitem[HS14]{MR3152942}
H.~A. Helfgott and {\'A}.~Seress.
\newblock On the diameter of permutation groups.
\newblock {\em Ann. of Math. (2)}, 179(2):611--658, 2014.

\bibitem[{Jor}70]{zbMATH02721437}
C.~{Jordan}.
\newblock {Trait\'e des substitutions alg\'ebriques.}
\newblock {Paris}, 1870.

\bibitem[Kes59]{MR0109367}
H.~Kesten.
\newblock Symmetric random walks on groups.
\newblock {\em Trans. Amer. Math. Soc.}, 92:336--354, 1959.

\bibitem[Kow13]{MR3144176}
E.~Kowalski.
\newblock Explicit growth and expansion for {${\rm SL}_2$}.
\newblock {\em Int. Math. Res. Not. IMRN}, (24):5645--5708, 2013.

\bibitem[LP11]{LP}
M.~J. Larsen and R.~Pink.
\newblock Finite subgroups of algebraic groups.
\newblock {\em J. Amer. Math. Soc.}, 24(4):1105--1158, 2011.

\bibitem[LPW09]{MR2466937}
D.~A. Levin, Y.~Peres, and E.~L. Wilmer.
\newblock {\em Markov chains and mixing times}.
\newblock American Mathematical Society, Providence, RI, 2009.
\newblock With a chapter by James G. Propp and David B. Wilson.

\bibitem[{Moo}96]{zbMATH02675601}
E.~H. {Moore}.
\newblock {A doubly-infinite system of simple groups}.
\newblock {Chicago Congress, Mathem. papers. 208-242 (1896).}, 1896.

\bibitem[Mum99]{MR1748380}
D.~Mumford.
\newblock {\em The red book of varieties and schemes}, volume 1358 of {\em
  Lecture Notes in Mathematics}.
\newblock Springer-Verlag, Berlin, expanded edition, 1999.
\newblock Includes the Michigan lectures (1974) on curves and their Jacobians,
  With contributions by Enrico Arbarello.

\bibitem[MVW84]{MR735226}
C.~R. Matthews, L.~N. Vaserstein, and B.~Weisfeiler.
\newblock Congruence properties of {Z}ariski-dense subgroups. {I}.
\newblock {\em Proc. London Math. Soc. (3)}, 48(3):514--532, 1984.

\bibitem[Nor87]{MR880952}
M.~V. Nori.
\newblock On subgroups of {${\rm GL}_n({\bf F}_p)$}.
\newblock {\em Invent. math.}, 88(2):257--275, 1987.

\bibitem[NP11]{MR2800484}
N.~Nikolov and L.~Pyber.
\newblock Product decompositions of quasirandom groups and a {J}ordan type
  theorem.
\newblock {\em J. Eur. Math. Soc. (JEMS)}, 13(4):1063--1077, 2011.

\bibitem[Pet12]{MR3063158}
G.~Petridis.
\newblock New proofs of {P}l\"unnecke-type estimates for product sets in
  groups.
\newblock {\em Combinatorica}, 32(6):721--733, 2012.

\bibitem[Pl{\"u}70]{MR0266892}
H.~Pl{\"u}nnecke.
\newblock Eine zahlentheoretische {A}nwendung der {G}raphentheorie.
\newblock {\em J. Reine Angew. Math.}, 243:171--183, 1970.

\bibitem[PS]{PS}
L.~Pyber and E.~Szab\'o.
\newblock Growth in finite simple groups of {L}ie type.
\newblock To appear in {\em J. Amer. Math. Soc}.

\bibitem[RT85]{MR810596}
I.~Z. Ruzsa and S.~Turj{\'a}nyi.
\newblock A note on additive bases of integers.
\newblock {\em Publ. Math. Debrecen}, 32(1-2):101--104, 1985.

\bibitem[Ruz89]{MR2314377}
I.~Z. Ruzsa.
\newblock An application of graph theory to additive number theory.
\newblock {\em Sci. Ser. A Math. Sci. (N.S.)}, 3:97--109, 1989.

\bibitem[Ruz91]{MR1139055}
I.~Z. Ruzsa.
\newblock Arithmetic progressions in sumsets.
\newblock {\em Acta Arith.}, 60(2):191--202, 1991.

\bibitem[Sel65]{MR0182610}
A.~Selberg.
\newblock On the estimation of {F}ourier coefficients of modular forms.
\newblock In {\em Proc. {S}ympos. {P}ure {M}ath., {V}ol. {VIII}}, pages 1--15.
  Amer. Math. Soc., Providence, R.I., 1965.

\bibitem[Sha99]{MR1691549}
Y.~Shalom.
\newblock Expander graphs and amenable quotients.
\newblock In {\em Emerging applications of number theory ({M}inneapolis, {MN},
  1996)}, volume 109 of {\em IMA Vol. Math. Appl.}, pages 571--581. Springer,
  New York, 1999.

\bibitem[SX91]{SarnakXue}
P.~Sarnak and X.~X. Xue.
\newblock Bounds for multiplicities of automorphic representations.
\newblock {\em Duke Math. J.}, 64(1):207--227, 1991.

\bibitem[Tao08]{MR2501249}
T.~Tao.
\newblock Product set estimates for non-commutative groups.
\newblock {\em Combinatorica}, 28(5):547--594, 2008.

\bibitem[Tao15]{MR3309986}
T.~Tao.
\newblock {\em Expansion in finite simple groups of {L}ie type}, volume 164 of
  {\em Graduate Studies in Mathematics}.
\newblock American Mathematical Society, Providence, RI, 2015.

\bibitem[Var12]{MR2862040}
P.~P. Varj{\'u}.
\newblock Expansion in {${\rm SL}_d(\mathcal{O}_K/I)$}, {$I$} square-free.
\newblock {\em J. Eur. Math. Soc. (JEMS)}, 14(1):273--305, 2012.

\bibitem[Wei84]{MR763908}
B.~Weisfeiler.
\newblock Strong approximation for {Z}ariski-dense subgroups of semisimple
  algebraic groups.
\newblock {\em Ann. of Math. (2)}, 120(2):271--315, 1984.

\end{thebibliography}
\end{document}